\documentclass{ws-procs9x6tn}
\usepackage{epic,eepic} % for figures
\usepackage[mathscr]{eucal} %  use \EuScript (\mathcal unchanged)
\usepackage{amsmath}% use \align
\newtheorem{Theorem}{Theorem}[section]
\newtheorem{Lemma}[Theorem]{Lemma}
\newtheorem{Definition}[Theorem]{Definition}
\newtheorem{Example}[Theorem]{Example}
\newtheorem{Remark}[Theorem]{Remark}

\begin{document}
\title{T-systems, Y-systems, and cluster algebras:\\
Tamely laced case}

\author{Tomoki Nakanishi}

\address{
 Graduate School of Mathematics, Nagoya University,
Nagoya, 464-8604, Japan\\
E-mail: nakanisi@math.nagoya-u.ac.jp}

% I use \address because I do not have \dedication
\address{Dedicated to Professor Tetsuji Miwa on his 60th birthday}

\begin{abstract}
The T-systems and Y-systems are
classes of algebraic relations
originally associated with quantum affine algebras
and Yangians.
Recently they were generalized to
quantum affinizations
of quantum Kac-Moody algebras
associated with a wide class of
generalized Cartan matrices which
we say tamely laced.
Furthermore, in the simply laced case,
and also in the nonsimply laced case of finite type,
they were identified with
relations arising from cluster algebras.
In this note
we generalize such an  identification
to  any tamely laced Cartan matrices,
especially to the nonsimply laced ones
of nonfinite type.
\end{abstract}


\bodymatter

\section{Introduction}

The T-systems and Y-systems appear in various aspects
for integrable
systems.
Originally, the T-systems are systems of relations
among the Kirillov-Reshetikhin modules
in the Grothendieck rings of modules over
quantum affine algebras and Yangians.
The T and Y-systems are related to each other by certain changes
of variables.
See, for example, Ref.~\refcite{Inoue09} and references therein
for more information and background.

Let $I=\{1,\dots,r\}$
and let $C=(C_{ij})_{i,j \in I}$ be a 
{\em (generalized) Cartan matrix\/} in
Ref.~\refcite{Kac90};
namely, it satisfies
$C_{ij}\in \mathbb{Z}$, $C_{ii}=2$, $C_{ij}\leq 0$ for 
any $i\neq j$, and $C_{ij}=0$ if and only if $C_{ji}=0$.
We assume that $C$ is {\em symmetrizable\/},
i.e., there is a diagonal matrix $D=\mathrm{diag}
(d_1,\dots,d_r)$ with $d_i\in \mathbb{N}:=\mathbb{Z}_{> 0}$
such that $B=DC$ is symmetric.
We always assume that there is no common divisor
for $d_1,\dots,d_r$ except for 1.
Following Ref.~\refcite{Kuniba09}, we say that a
Cartan matrix $C$ is {\em tamely laced\/} if
it is symmetrizable and satisfies the following condition
due to Hernandez\cite{Hernandez07a}:
\begin{align}
\label{eq:Ccond1}
\mbox{If $C_{ij}< -1$, then $d_i=-C_{ji}=1$.}
\end{align}

Recently, the T-systems were generalized by Hernandez \cite{Hernandez07a}
to the quantum affinizations 
of the quantum Kac-Moody algebras
associated with {\em tamely laced\/} Cartan matrices.
Subsequently, the corresponding Y-systems were also introduced by
Kuniba, Suzuki, and the author \cite{Kuniba09}.

Remarkably, these T and Y-systems are identified with
(a part of) relations among the variables
for {\em cluster algebras}\cite{Fomin02,Fomin07},
which are a class of commutative algebras
closely related to the representation theory of quivers.
For the T and Y-systems associated with
{\em simply laced\/} Cartan matrices  of {\em finite type},
this identification is a topic intensively studied by various authors
recently with several reasons (periodicity, categorification,
positivity, dilogarithm identities, etc.)
 \cite{Fomin03b,Fomin07,DiFrancesco09a,Keller08,
Inoue09,Hernandez09,Nakajima09,DiFrancesco09b,Nakanishi09,
Keller10,Inoue10a,
Inoue10b}.
In Ref.~\refcite{Kuniba09}, 
such an identification
 was generalized to the {\em simply laced\/}
 Cartan matrices.
In Refs.~\refcite{Inoue10a} and \refcite{Inoue10b},
it was  also extended  to the {\em nonsimply laced\/} Cartan matrices
of {\em finite type}.

In this note we present a generalization of the above identification
to  {\em any tamely laced\/} Cartan matrices,
especially to the {\em nonsimply laced ones of
nonfinite type},
thereby justifying Sec.~6.5 of Ref.~\refcite{Kuniba09}
which announced that such a generalization is possible.
Basically it is a straightforward extension of the simply laced ones
\cite{Kuniba09} and the nonsimply laced ones of finite type
\cite{Inoue10a,Inoue10b}, but it is necessarily more complicated.
At this time we do not have any immediate application of such a
generalization.
However, we believe that this is a necessary step
toward further study of the intriguing interplay
of two worlds --- the representation theories of 
quantum groups and quivers --- through cluster algebras.

\section{T and Y-systems}

In this section we recall  the definitions
of (restricted) T and Y-systems.
See Ref.~\refcite{Kuniba09} for more detail.

With a tamely laced Cartan matrix $C$,
we  associate a {\em Dynkin diagram\/} $X(C)$ in the standard way:
For any pair $i\neq j \in I$ with $C_{ij}<0$,
the vertices $i$ and $j$ are connected by
 $\max\{|C_{ij}|,|C_{ji}|\}$ lines,
and the lines are equipped with an arrow from $j$ to $i$
 if $C_{ij}<-1$.
Note that the condition \eqref{eq:Ccond1} means

\begin{itemize}
\item[(i)] the vertices $i$ and $j$ are not connected,
if $d_i,d_j>1$ and $d_i\neq d_j$,

\item[(ii)]
the vertices $i$ and $j$ are connected by $d_i$ lines with
 an arrow from $i$ to $j$
or not connected, if $d_i>1$ and $d_j=1$,

\item[(iii)]
the vertices $i$ and $j$ are connected by a single line
or not connected, if $d_i=d_j$.
\end{itemize}

As usual,
 we say that a  Cartan matrix $C$ is {\em simply laced\/} if
$C_{ij}=0$ or $-1$ for any $i\neq j$.
If $C$ is simply laced, then it is tamely laced.

For a tamely laced Cartan matrix $C$, we set  integers $t$
and $t_a$ ($a\in I$) by
\begin{align}
\label{eq:t1}
t=\mathrm{lcm}(d_1,\dots,d_r),
\quad t_a=\frac{t}{d_a}.
\end{align}
For an integer $\ell\geq 2$,
we set
\begin{align}
\mathcal{I}_{\ell}:=\{(a,m,u)\mid
a\in I; m=1,\dots,t_a\ell-1; u\in 
\frac{1}{t}\mathbb{Z}\}.
\end{align}
For $a,b\in I$, we write $a\sim b$ if
$C_{ab}<0$, i.e., $a$ and $b$ are adjacent 
in $X(C)$.

First, we introduce the T-systems and the associated rings.

\begin{Definition}
Fix an integer $\ell \geq 2$.
For a tamely laced  Cartan matrix $C$,
the {\em level $\ell$ restricted T-system $\mathbb{T}_{\ell}(C)$
 associated with $C$ (with the unit boundary condition)}
is the following system of relations for
a family of variables
 $T_{\ell}=\{T^{(a)}_m(u) \mid (a,m,u)\in \mathcal{I}_{\ell} \}$,
\begin{align}
\label{eq:T1}
\textstyle
T^{(a)}_m\left(u-\frac{d_a}{t}\right)
T^{(a)}_m\left(u+\frac{d_a}{t}\right)
&=
T^{(a)}_{m-1}(u)T^{(a)}_{m+1}(u)
+
\prod_{b: b\sim a}
T^{(b)}_{\frac{d_a}{d_b}m}(u)
\quad\mbox{if $d_a>1$},
\\
\label{eq:T2}
\textstyle
T^{(a)}_m\left(u-\frac{d_a}{t}\right)
T^{(a)}_m\left(u+\frac{d_a}{t}\right)
&=
T^{(a)}_{m-1}(u)T^{(a)}_{m+1}(u)
+
\prod_{b: b\sim a}
S^{(b)}_{m}(u)
\quad\mbox{if $d_a=1$},
\end{align}
where 
$T^{(a)}_0 (u)= 1$,
and furthermore, $T^{(a)}_{t_a\ell}(u)=1$
(the {\em unit boundary condition})
 if they occur
in the right hand sides in the relations.
The symbol $S^{(b)}_{m}(u)$ is defined 
as follows. For $m=0,1,2,\dots$ and $0\leq j < d_b$,
\begin{align}
\begin{split}
\textstyle
S^{(b)}_{d_bm+j}(u)
=&
\left\{
{\displaystyle
\prod_{k=1}^j
}
 T^{(b)}_{m+1}
\left(u+\frac{1}{t}(j+1-2k)\right)
\right\}\\
&
\times
\left\{
{\displaystyle
\prod_{k=1}^{d_b-j}
}
 T^{(b)}_{m}\left(u+\frac{1}{t}(d_b-j+1-2k)\right)
\right\}.
\end{split}
\end{align}
\end{Definition}

For the later use, let us formally write \eqref{eq:T1} and \eqref{eq:T2}
 in a unified manner
\begin{align}
\label{eq:Tu}
\begin{split}
T^{(a)}_{m}\left(u-\textstyle\frac{d_a}{t}\right)
T^{(a)}_{m}\left(u+\textstyle\frac{d_a}{t}\right)
&=
T^{(a)}_{m-1}(u)T^{(a)}_{m+1}(u)\\
&\quad +
\prod_{(b,k,v)\in \mathcal{I}_{\ell}}
T^{(b)}_{k}(v)^{G(b,k,v;a,m,u)}.
\end{split}
\end{align}

\begin{Definition}
\label{defn:TC}
Let $\EuScript{T}_{\ell}(C)$
be the commutative ring over $\mathbb{Z}$ 
with identity element,  with generators
$T^{(a)}_m(u)^{\pm 1}$ ($(a,m,u)\in \mathcal{I}_{\ell}$)
and relations $\mathbb{T}_{\ell}(C)$
together with $T^{(a)}_m(u)T^{(a)}_m(u)^{-1}=1$.
Let $\EuScript{T}^{\circ}_{\ell}(C)$
be the subring of $\EuScript{T}_{\ell}(C)$
generated by 
$T^{(a)}_m(u)$ ($(a,m,u)\in \mathcal{I}_{\ell}$).
\end{Definition}

Similarly, we introduce the Y-systems
and the associated groups.

\begin{Definition}
Fix an integer $\ell \geq 2$.
For a tamely laced Cartan matrix $C$,
the {\em level $\ell$
 restricted Y-system $\mathbb{Y}_{\ell}(C)$ associated with $C$}
is the following system of relations for
a family of variables
 $Y=\{Y^{(a)}_m(u) \mid (a,m,u)\in \mathcal{I}_{\ell} \}$,
\begin{align}
\label{eq:Y1}
\textstyle
Y^{(a)}_m\left(u-\frac{d_a}{t}\right)
Y^{(a)}_m\left(u+\frac{d_a}{t}\right)
&=
\frac{
{\displaystyle \prod_{b:b\sim a}}
Z^{(b)}_{\frac{d_a}{d_b},m}(u)
}
{
(1+Y^{(a)}_{m-1}(u)^{-1})(1+Y^{(a)}_{m+1}(u)^{-1})}
\quad\mbox{if $d_a>1$},
\\
\label{eq:Y2}
\textstyle
Y^{(a)}_m\left(u-\frac{d_a}{t}\right)
Y^{(a)}_m\left(u+\frac{d_a}{t}\right)
&=
\frac{
{\displaystyle \prod_{b:b\sim a}}
\left(1+Y^{(b)}_{\frac{m}{d_b}}(u)\right)
}
{
(1+Y^{(a)}_{m-1}(u)^{-1})(1+Y^{(a)}_{m+1}(u)^{-1})}
\quad\mbox{if $d_a=1$},
\end{align}
where 
$Y^{(a)}_0 (u)^{-1}=
Y^{(a)}_{t_a\ell} (u)^{-1}=0$ if they occur
in the right hand sides in the relations.
Besides, 
$Y^{(b)}_{m/d_b}(u)=0$ in \eqref{eq:Y2}
if $m/d_b \not\in \mathbb{N}$.
The symbol $Z^{(b)}_{p,m}(u)$ ($p\in \mathbb{N}$)
is defined as follows.
\begin{align}
\textstyle
Z^{(b)}_{p,m}(u)=
{\displaystyle\prod_{j=-p+1}^{p-1}}
\left\{
{\displaystyle\prod_{k=1}^{p-|j|}}
\left(
1+Y^{(b)}_{pm+j}\bigl(
u+\frac{1}{t}(p-|j|+1-2k)
\bigr)
\right)
\right\}.
\end{align}
\end{Definition}

One can write \eqref{eq:Y1} and \eqref{eq:Y2}
in a unified manner as
\begin{align}
\textstyle
Y^{(a)}_m\left(u-\frac{d_a}{t}\right)
Y^{(a)}_m\left(u+\frac{d_a}{t}\right)
&=
\frac{
{\displaystyle \prod_{(b,k,v)\in \mathcal{I}_{\ell}}
(1+Y^{(b)}_k(v))^{{}^t G(b,k,v;a,m,u)}
}
}
{
(1+Y^{(a)}_{m-1}(u)^{-1})(1+Y^{(a)}_{m+1}(u)^{-1})},
\end{align}
where ${}^t G(b,k,v;a,m,u):=G(a,m,u;b,k,v)$.

A {\em semifield\/} $(\mathbb{P},\oplus)$ is an
abelian multiplicative group $\mathbb{P}$ endowed with a binary
operation of addition $\oplus$ which is commutative,
associative, and distributive with respect to the
multiplication in $\mathbb{P}$.

\begin{Definition}
\label{def:YC}
Let $\EuScript{Y}_{\ell}(C)$
be the semifield with generators
$Y^{(a)}_m(u)$
 $((a,m,u)\in \mathcal{I}_{\ell})$
and relations $\mathbb{Y}_{\ell}(C)$.
Let $\EuScript{Y}^{\circ}_{\ell}(C)$
be the multiplicative subgroup
of $\EuScript{Y}_{\ell}(C)$
generated by
$Y^{(a)}_m(u)$, $1+Y^{(a)}_m(u)$
 ($(a,m,u)\in \mathcal{I}_{\ell}$).
(Here we use the symbol $+$ instead of $\oplus$ 
for simplicity.)
\end{Definition}

\section{Cluster algebra with coefficients}
\label{sec:cluster}

In this section
we recall the definition
of cluster algebras with
coefficients
following Ref.~\refcite{Fomin07}.
The description here is minimal to fix convention
and notion.
See Ref.~\refcite{Fomin07} for more detail and information.

Let $I$ be a finite set, and let $B=(B_{ij})_{i,j\in I}$ be a
 skew
symmetric (integer) matrix.
Let $x=(x_i)_{i\in I}$ and $y=(y_i)_{i\in I}$
be $I$-tuples of formal variables.
Let  $\mathbb{P}=\mathbb{Q}_{\mathrm{sf}}(y)$ 
be the {\em universal semifield\/} of
$y=(y_i)_{i\in I}$, namely,
the semifield consisting of 
the {\em subtraction-free\/} rational functions of $y$ with
usual multiplication and addition (but no subtraction)
in the rational function
 field $\mathbb{Q}(y)$.
Let $\mathbb{Q}\mathbb{P}$
denote the quotient field of the group ring $\mathbb{Z}\mathbb{P}$
of $\mathbb{P}$.

For the above triplet $(B,x,y)$, called the {\em initial seed},
the  {\em cluster algebra $\mathcal{A}(B,x,y)$ with
coefficients in $\mathbb{P}$} is defined as follows.

Let $(B',x',y')$ be a triplet consisting of
skew symmetric matrix $B'$,
an $I$-tuple $x'=(x'_i)_{i\in I}$ with
 $x'_i\in \mathbb{Q}\mathbb{P}(x)$,
and 
an $I$-tuple $y'=(y'_i)_{i\in I}$ with $y'_i\in \mathbb{P}$.
For each $k\in I$, we define another triplet
$(B'',x'',y'')=\mu_k(B',x',y')$, called the {\em mutation
of $(B',x',y')$ at $k$}, as follows.

{\it  (i) Mutations of matrix.}
\begin{align}
\label{eq:Bmut}
B''_{ij}=
\begin{cases}
-B'_{ij}& \mbox{$i=k$ or $j=k$},\\
B'_{ij}+\frac{1}{2}
(|B'_{ik}|B'_{kj} + B'_{ik}|B'_{kj}|)
&\mbox{otherwise}.
\end{cases}
\end{align}

{\it (ii)  Exchange relation of coefficient tuple.}
\begin{align}
\label{eq:coef}
y''_i =
\begin{cases}
\displaystyle
{y'_k}{}^{-1}&i=k,\\
\displaystyle
y'_i \left(\frac{y'_k}{1\oplus {y'_k}}\right)^{B'_{ki}}&
i\neq k,\ B'_{ki}\geq 0,\\
y'_i (1\oplus y'_k)^{-B'_{ki}}&
i\neq k,\ B'_{ki}\leq 0.\\
\end{cases}
\end{align}

{\it (iii)   Exchange relation of cluster.}
\begin{align}
\label{eq:clust}
x''_i =
\begin{cases}
\displaystyle
\frac{y'_k
\prod_{j: B'_{jk}>0} {x'_j}^{B'_{jk}}
+
\prod_{j: B'_{jk}<0} {x'_j}^{-B'_{jk}}
}{(1\oplus y'_k)x'_k}
&
i= k,\\
{x'_i}&i\neq k.\\
\end{cases}
\end{align}
It is easy to see that  $\mu_k$ is an involution,
namely, $\mu_k(B'',x'',y'')=(B',x',y')$.
Now, starting from the initial seed
$(B,x,y)$, iterate mutations and collect all the
resulted triplets $(B',x',y')$.
We call $(B',x',y')$ the {\em seeds\/},
$y'$ and $y'_i$ a {\em coefficient tuple} and
a {\em coefficient\/},
$x'$  and $x'_i$, a {\em cluster\/} and
a {\em cluster variable}, respectively.
The {\em cluster algebra $\mathcal{A}(B,x,y)$ with
coefficients in $\mathbb{P}$} is the
$\mathbb{Z}\mathbb{P}$-subalgebra of the
rational function field $\mathbb{Q}\mathbb{P}(x)$
generated by all the cluster variables.
Similarly, the {\em coefficient group $\mathcal{G}(B,y)$
associated with $\mathcal{A}(B,x,y)$}
is the
multiplicative subgroup of the semifield $\mathbb{P}$
generated by all the coefficients $y'_i$ together with $1\oplus y'_i$.

It is standard to identify
a skew-symmetric (integer) matrix $B=(B_{ij})_{i,j\in I}$
with a {\em quiver $Q$
without loops or 2-cycles}.
The set of the vertices of $Q$ is given by $I$,
and we put $B_{ij}$ arrows from $i$ to $j$ 
if $B_{ij}>0$.
The mutation $Q''=\mu_k(Q')$ of a quiver $Q'$ is given by the following
rule:
For each pair of an incoming arrow $i\rightarrow k$
and an outgoing arrow $k\rightarrow j$ in $Q'$,
add a new arrow $i\rightarrow j$.
Then, remove a maximal set of pairwise disjoint 2-cycles.
Finally, reverse all arrows incident with $k$.

\section{Cluster algebraic formulation: The case $|I|=2$; $t$ is odd}
\label{sect:todd}

\subsection{Cartan matrix $M_t$}

We are going to identify $\mathbb{T}_{\ell}(C)$ and
$\mathbb{Y}_{\ell}(C)$ as relations for cluster variables
and coefficients of the cluster algebra
associated with a certain quiver $Q_{\ell}(C)$.

To begin with, we consider the case
$I=\{1,2\}$, which will be used as building blocks of 
the general case.
Without loss of generality we may assume that
a Cartan matrix $C$ is {\em indecomposable},
i.e., $X(C)$ is connected.
Thus, we assume that our tamely laced Cartan matrix $C$
has the form ($t=1,2,\dots$)
\begin{align}
C=M_t:=
\begin{pmatrix}
2 & -1\\
-t & 2\\
\end{pmatrix},
\quad
D=
\begin{pmatrix}
t & 0\\
0 & 1\\
\end{pmatrix}.
\end{align}
We have the data $d_1=t$, $d_2=1$,
$t=\mathrm{lcm}(d_1,d_2)$, $t_1=1$, $t_2=t$,
and the corresponding Dynkin diagram looks as follows
($t$ lines in the middle and there is no arrow for $t=1$):
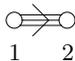
\begin{align*}
\begin{picture}(20,25)(0,-15)
%
% A_r
%
% B_r
\put(0,0){
\put(0,0){\circle{6}}
\put(20,0){\circle{6}}
\drawline(2,-2)(18,-2)
\drawline(2,2)(18,2)
\drawline(3,0)(17,0)
\drawline(7,6)(13,0)
\drawline(7,-6)(13,0)
\put(-2,-15){\small $1$}
\put(18,-15){\small $2$}
}
\end{picture}
\end{align*}
We ask the reader to
 refer to Refs.~\refcite{Kuniba09}, \refcite{Inoue10a},
and \refcite{Inoue10b},
where the cases $t=1$ (type $A_2$),
 $2$ (type $B_2$), and  $3$ (type $G_2$),
respectively, are treated in detail.

It turns out that
we should separate the problem depending on the parity 
of $t$.
In this section we consider the case when $t$ is odd.

\subsection{Parity decompositions of T and Y-systems}

For a triplet $(a,m,u)\in \mathcal{I}_{\ell}$,
we set the parity conditions $\mathbf{P}_{+}$ and
$\mathbf{P}_{-}$ by
\begin{align}
\label{eq:GPcond}
\begin{split}
\mathbf{P}_{+}:& \ \mbox{$m+tu$ is odd for $a=1$; $m+tu$ is even for $a=2$},\\
\mathbf{P}_{-}:& \ \mbox{$m+tu$ is even for $a=1$; $m+tu$ is odd for $a=2$}.
\end{split}
\end{align}
We write, for example, $(a,m,u):\mathbf{P}_{+}$ if $(a,m,u)$ satisfies
$\mathbf{P}_{+}$.
We have $\mathcal{I}_{\ell}=
\mathcal{I}_{\ell+}\sqcup \mathcal{I}_{\ell-}$,
where $\mathcal{I}_{\ell\varepsilon}$ 
is the set of all $(a,m,u):\mathbf{P}_{\varepsilon}$.
Define $\EuScript{T}^{\circ}_{\ell}(M_t)_{\varepsilon}$
($\varepsilon=\pm$)
to be the subring of $\EuScript{T}^{\circ}_{\ell}(M_t)$
generated by
 $T^{(a)}_m(u)$
$((a,m,u)\in \mathcal{I}_{\ell\varepsilon})$.
Then, we have
$\EuScript{T}^{\circ}_{\ell}(M_t)_+
\simeq
\EuScript{T}^{\circ}_{\ell}(M_t)_-
$
by $T^{(a)}_m(u)\mapsto T^{(a)}_m(u+\frac{1}{t})$ and
\begin{align}
\label{eq:Tfact}
\EuScript{T}^{\circ}_{\ell}(M_t)
\simeq
\EuScript{T}^{\circ}_{\ell}(M_t)_+
\otimes_{\mathbb{Z}}
\EuScript{T}^{\circ}_{\ell}(M_t)_-.
\end{align}

For a triplet $(a,m,u)\in \mathcal{I}_{\ell}$,
we introduce another parity conditions $\mathbf{P}'_{+}$ and
$\mathbf{P}'_{-}$ by
\begin{align}
\label{eq:GP'cond}
\begin{split}
\mathbf{P}'_{+}:& \ \mbox{$m+tu$ is even for $a=1$; $m+tu$ is odd for $a=2$},\\
\mathbf{P}'_{-}:& \ \mbox{$m+tu$ is odd for $a=1$; $m+tu$ is even for $a=2$}.
\end{split}
\end{align}
Since $\mathbf{P}'_{\pm}=\mathbf{P}_{\mp}$,
it may seem redundant,
but we use this notation to make the description unified 
for both odd and even $t$.
We have
\begin{align}
\label{eq:PP'}
(a,m,u):\mathbf{P}'_+ \ \Longleftrightarrow\ 
\textstyle (a,m,u\pm \frac{d_a}{t}):\mathbf{P}_+.
\end{align}
Let $\mathcal{I}'_{\ell\varepsilon}$ be
 the set of all $(a,m,u):\mathbf{P}'_{\varepsilon}$.
Define $\EuScript{Y}^{\circ}_{\ell}(M_t)_{\varepsilon}$
($\varepsilon=\pm$)
to be the subgroup of $\EuScript{Y}^{\circ}_{\ell}(M_t)$
generated by
$Y^{(a)}_m(u)$, $1+Y^{(a)}_m(u)$
$((a,m,u)\in \mathcal{I}'_{\ell\varepsilon})$.
Then, we have
$\EuScript{Y}^{\circ}_{\ell}(M_t)_+
\simeq
\EuScript{Y}^{\circ}_{\ell}(M_t)_-
$
by $Y^{(a)}_m(u)\mapsto Y^{(a)}_m(u+\frac{1}{t})$,
$1+Y^{(a)}_m(u)\mapsto 1+Y^{(a)}_m(u+\frac{1}{t})$,
 and
\begin{align}
\label{eq:Yfact}
\EuScript{Y}^{\circ}_{\ell}(M_t)
\simeq
\EuScript{Y}^{\circ}_{\ell}(M_t)_+
\times
\EuScript{Y}^{\circ}_{\ell}(M_t)_-.
\end{align}

\subsection{Quiver $Q_{\ell}(M_t)$}
\label{subsect:quiverodd}

%%%%%%%%%%%%%%%%%%%%%%%%%%%%%
% G_2
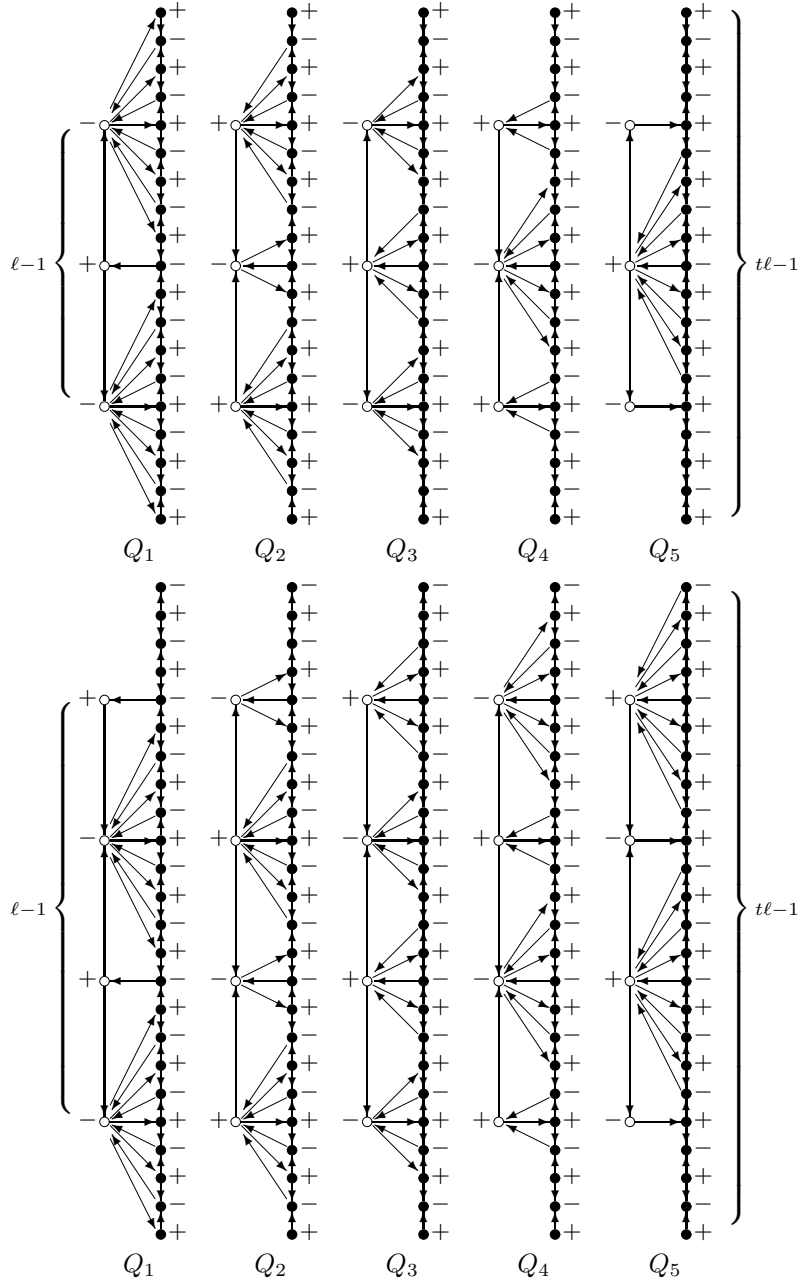
\begin{figure}
\setlength{\unitlength}{0.71pt}
\begin{picture}(360,288)(-60,-15)
%
% first diagram
\put(0,0)
{
\put(0,60){\circle{5}}
\put(0,135){\circle{5}}
\put(0,210){\circle{5}}
\put(30,0){\circle*{5}}
\put(30,15){\circle*{5}}
\put(30,30){\circle*{5}}
\put(30,45){\circle*{5}}
\put(30,60){\circle*{5}}
\put(30,75){\circle*{5}}
\put(30,90){\circle*{5}}
\put(30,105){\circle*{5}}
\put(30,120){\circle*{5}}
\put(30,135){\circle*{5}}
\put(30,150){\circle*{5}}
\put(30,165){\circle*{5}}
\put(30,180){\circle*{5}}
\put(30,195){\circle*{5}}
\put(30,210){\circle*{5}}
\put(30,225){\circle*{5}}
\put(30,240){\circle*{5}}
\put(30,255){\circle*{5}}
\put(30,270){\circle*{5}}
%
% vertical arrows
\put(30,3){\vector(0,1){9}}
\put(30,27){\vector(0,-1){9}}
\put(30,33){\vector(0,1){9}}
\put(30,57){\vector(0,-1){9}}
\put(30,63){\vector(0,1){9}}
\put(30,87){\vector(0,-1){9}}
\put(30,93){\vector(0,1){9}}
\put(30,117){\vector(0,-1){9}}
\put(30,123){\vector(0,1){9}}
\put(30,147){\vector(0,-1){9}}
\put(30,153){\vector(0,1){9}}
\put(30,177){\vector(0,-1){9}}
\put(30,183){\vector(0,1){9}}
\put(30,207){\vector(0,-1){9}}
\put(30,213){\vector(0,1){9}}
\put(30,237){\vector(0,-1){9}}
\put(30,243){\vector(0,1){9}}
\put(30,267){\vector(0,-1){9}}
\put(0,132){\vector(0,-1){69}}
\put(0,138){\vector(0,1){69}}
% diagonal arrows
\put(3,51){\vector(1,-2){24}}
\put(27,19){\vector(-2,3){23}}
\put(3,58){\vector(1,-1){24}}
\put(27,47){\vector(-2,1){23}}
\put(3,60){\vector(1,0){24}}
\put(27,73){\vector(-2,-1){23}}
\put(3,62){\vector(1,1){24}}
\put(27,101){\vector(-2,-3){23}}
\put(3,69){\vector(1,2){24}}
\put(27,135){\vector(-1,0){24}}
\put(3,201){\vector(1,-2){24}}
\put(27,169){\vector(-2,3){23}}
\put(3,208){\vector(1,-1){24}}
\put(27,197){\vector(-2,1){23}}
\put(3,210){\vector(1,0){24}}
\put(27,223){\vector(-2,-1){23}}
\put(3,212){\vector(1,1){24}}
\put(27,251){\vector(-2,-3){23}}
\put(3,219){\vector(1,2){24}}
\put(3,-2)
{
\put(-17,60){\small $-$}
\put(-17,135){\small $+$}
\put(-17,210){\small $-$}
}
\put(4,-2)
{
\put(30,0){\small $+$}
\put(30,15){\small $-$}
\put(30,30){\small $+$}
\put(30,45){\small $-$}
\put(30,60){\small $+$}
\put(30,75){\small $-$}
\put(30,90){\small $+$}
\put(30,105){\small $-$}
\put(30,120){\small $+$}
\put(30,135){\small $-$}
\put(30,150){\small $+$}
\put(30,165){\small $-$}
\put(30,180){\small $+$}
\put(30,195){\small $-$}
\put(30,210){\small $+$}
\put(30,225){\small $-$}
\put(30,240){\small $+$}
\put(30,255){\small $-$}
\put(30,270){\small $+$}
}
}
% second diagram
\put(70,0)
{
\put(0,60){\circle{5}}
\put(0,135){\circle{5}}
\put(0,210){\circle{5}}
\put(30,0){\circle*{5}}
\put(30,15){\circle*{5}}
\put(30,30){\circle*{5}}
\put(30,45){\circle*{5}}
\put(30,60){\circle*{5}}
\put(30,75){\circle*{5}}
\put(30,90){\circle*{5}}
\put(30,105){\circle*{5}}
\put(30,120){\circle*{5}}
\put(30,135){\circle*{5}}
\put(30,150){\circle*{5}}
\put(30,165){\circle*{5}}
\put(30,180){\circle*{5}}
\put(30,195){\circle*{5}}
\put(30,210){\circle*{5}}
\put(30,225){\circle*{5}}
\put(30,240){\circle*{5}}
\put(30,255){\circle*{5}}
\put(30,270){\circle*{5}}
%
% vertical arrows
\put(30,3){\vector(0,1){9}}
\put(30,27){\vector(0,-1){9}}
\put(30,33){\vector(0,1){9}}
\put(30,57){\vector(0,-1){9}}
\put(30,63){\vector(0,1){9}}
\put(30,87){\vector(0,-1){9}}
\put(30,93){\vector(0,1){9}}
\put(30,117){\vector(0,-1){9}}
\put(30,123){\vector(0,1){9}}
\put(30,147){\vector(0,-1){9}}
\put(30,153){\vector(0,1){9}}
\put(30,177){\vector(0,-1){9}}
\put(30,183){\vector(0,1){9}}
\put(30,207){\vector(0,-1){9}}
\put(30,213){\vector(0,1){9}}
\put(30,237){\vector(0,-1){9}}
\put(30,243){\vector(0,1){9}}
\put(30,267){\vector(0,-1){9}}
\put(0,63){\vector(0,1){69}}
\put(0,207){\vector(0,-1){69}}
% diagonal arrows
\put(27,19){\vector(-2,3){23}}
\put(3,58){\vector(1,-1){24}}
\put(27,47){\vector(-2,1){23}}
\put(3,60){\vector(1,0){24}}
\put(27,73){\vector(-2,-1){23}}
\put(3,62){\vector(1,1){24}}
\put(27,101){\vector(-2,-3){23}}
\put(3,133){\vector(2,-1){24}}
\put(27,135){\vector(-1,0){23}}
\put(3,137){\vector(2,1){24}}
\put(27,169){\vector(-2,3){23}}
\put(3,208){\vector(1,-1){24}}
\put(27,197){\vector(-2,1){23}}
\put(3,210){\vector(1,0){24}}
\put(27,223){\vector(-2,-1){23}}
\put(3,212){\vector(1,1){24}}
\put(27,251){\vector(-2,-3){23}}
\put(3,-2)
{
\put(-17,60){\small $+$}
\put(-17,135){\small $-$}
\put(-17,210){\small $+$}
}
\put(4,-2)
{
\put(30,0){\small $+$}
\put(30,15){\small $-$}
\put(30,30){\small $+$}
\put(30,45){\small $-$}
\put(30,60){\small $+$}
\put(30,75){\small $-$}
\put(30,90){\small $+$}
\put(30,105){\small $-$}
\put(30,120){\small $+$}
\put(30,135){\small $-$}
\put(30,150){\small $+$}
\put(30,165){\small $-$}
\put(30,180){\small $+$}
\put(30,195){\small $-$}
\put(30,210){\small $+$}
\put(30,225){\small $-$}
\put(30,240){\small $+$}
\put(30,255){\small $-$}
\put(30,270){\small $+$}
}
}
% third diagram
\put(140,0)
{
\put(0,60){\circle{5}}
\put(0,135){\circle{5}}
\put(0,210){\circle{5}}
\put(30,0){\circle*{5}}
\put(30,15){\circle*{5}}
\put(30,30){\circle*{5}}
\put(30,45){\circle*{5}}
\put(30,60){\circle*{5}}
\put(30,75){\circle*{5}}
\put(30,90){\circle*{5}}
\put(30,105){\circle*{5}}
\put(30,120){\circle*{5}}
\put(30,135){\circle*{5}}
\put(30,150){\circle*{5}}
\put(30,165){\circle*{5}}
\put(30,180){\circle*{5}}
\put(30,195){\circle*{5}}
\put(30,210){\circle*{5}}
\put(30,225){\circle*{5}}
\put(30,240){\circle*{5}}
\put(30,255){\circle*{5}}
\put(30,270){\circle*{5}}
%
% vertical arrows
\put(30,3){\vector(0,1){9}}
\put(30,27){\vector(0,-1){9}}
\put(30,33){\vector(0,1){9}}
\put(30,57){\vector(0,-1){9}}
\put(30,63){\vector(0,1){9}}
\put(30,87){\vector(0,-1){9}}
\put(30,93){\vector(0,1){9}}
\put(30,117){\vector(0,-1){9}}
\put(30,123){\vector(0,1){9}}
\put(30,147){\vector(0,-1){9}}
\put(30,153){\vector(0,1){9}}
\put(30,177){\vector(0,-1){9}}
\put(30,183){\vector(0,1){9}}
\put(30,207){\vector(0,-1){9}}
\put(30,213){\vector(0,1){9}}
\put(30,237){\vector(0,-1){9}}
\put(30,243){\vector(0,1){9}}
\put(30,267){\vector(0,-1){9}}
\put(0,132){\vector(0,-1){69}}
\put(0,138){\vector(0,1){69}}
% diagonal arrows
\put(3,58){\vector(1,-1){24}}
\put(27,47){\vector(-2,1){23}}
\put(3,60){\vector(1,0){24}}
\put(27,73){\vector(-2,-1){23}}
\put(3,62){\vector(1,1){24}}
\put(27,107){\vector(-1,1){23}}
\put(3,133){\vector(2,-1){24}}
\put(27,135){\vector(-1,0){23}}
\put(3,137){\vector(2,1){24}}
\put(27,163){\vector(-1,-1){23}}
\put(3,208){\vector(1,-1){24}}
\put(27,197){\vector(-2,1){23}}
\put(3,210){\vector(1,0){24}}
\put(27,223){\vector(-2,-1){23}}
\put(3,212){\vector(1,1){24}}
\put(3,-2)
{
\put(-17,60){\small $-$}
\put(-17,135){\small $+$}
\put(-17,210){\small $-$}
}
\put(4,-2)
{
\put(30,0){\small $+$}
\put(30,15){\small $-$}
\put(30,30){\small $+$}
\put(30,45){\small $-$}
\put(30,60){\small $+$}
\put(30,75){\small $-$}
\put(30,90){\small $+$}
\put(30,105){\small $-$}
\put(30,120){\small $+$}
\put(30,135){\small $-$}
\put(30,150){\small $+$}
\put(30,165){\small $-$}
\put(30,180){\small $+$}
\put(30,195){\small $-$}
\put(30,210){\small $+$}
\put(30,225){\small $-$}
\put(30,240){\small $+$}
\put(30,255){\small $-$}
\put(30,270){\small $+$}
}
}
%
% fourth diagram
\put(210,0)
{
\put(0,60){\circle{5}}
\put(0,135){\circle{5}}
\put(0,210){\circle{5}}
\put(30,0){\circle*{5}}
\put(30,15){\circle*{5}}
\put(30,30){\circle*{5}}
\put(30,45){\circle*{5}}
\put(30,60){\circle*{5}}
\put(30,75){\circle*{5}}
\put(30,90){\circle*{5}}
\put(30,105){\circle*{5}}
\put(30,120){\circle*{5}}
\put(30,135){\circle*{5}}
\put(30,150){\circle*{5}}
\put(30,165){\circle*{5}}
\put(30,180){\circle*{5}}
\put(30,195){\circle*{5}}
\put(30,210){\circle*{5}}
\put(30,225){\circle*{5}}
\put(30,240){\circle*{5}}
\put(30,255){\circle*{5}}
\put(30,270){\circle*{5}}
%
% vertical arrows
\put(30,3){\vector(0,1){9}}
\put(30,27){\vector(0,-1){9}}
\put(30,33){\vector(0,1){9}}
\put(30,57){\vector(0,-1){9}}
\put(30,63){\vector(0,1){9}}
\put(30,87){\vector(0,-1){9}}
\put(30,93){\vector(0,1){9}}
\put(30,117){\vector(0,-1){9}}
\put(30,123){\vector(0,1){9}}
\put(30,147){\vector(0,-1){9}}
\put(30,153){\vector(0,1){9}}
\put(30,177){\vector(0,-1){9}}
\put(30,183){\vector(0,1){9}}
\put(30,207){\vector(0,-1){9}}
\put(30,213){\vector(0,1){9}}
\put(30,237){\vector(0,-1){9}}
\put(30,243){\vector(0,1){9}}
\put(30,267){\vector(0,-1){9}}
\put(0,63){\vector(0,1){69}}
\put(0,207){\vector(0,-1){69}}
%\put(0,132){\vector(0,-1){69}}
%\put(0,138){\vector(0,1){69}}
% diagonal arrows
\put(27,47){\vector(-2,1){23}}
\put(3,60){\vector(1,0){24}}
\put(27,73){\vector(-2,-1){23}}
\put(3,128){\vector(2,-3){24}}
\put(27,107){\vector(-1,1){23}}
\put(3,133){\vector(2,-1){24}}
\put(27,135){\vector(-1,0){23}}
\put(3,137){\vector(2,1){24}}
\put(3,142){\vector(2,3){23}}
\put(27,163){\vector(-1,-1){24}}
\put(27,197){\vector(-2,1){23}}
\put(3,210){\vector(1,0){24}}
\put(27,223){\vector(-2,-1){23}}
\put(3,-2)
{
\put(-17,60){\small $+$}
\put(-17,135){\small $-$}
\put(-17,210){\small $+$}
}
\put(4,-2)
{
\put(30,0){\small $+$}
\put(30,15){\small $-$}
\put(30,30){\small $+$}
\put(30,45){\small $-$}
\put(30,60){\small $+$}
\put(30,75){\small $-$}
\put(30,90){\small $+$}
\put(30,105){\small $-$}
\put(30,120){\small $+$}
\put(30,135){\small $-$}
\put(30,150){\small $+$}
\put(30,165){\small $-$}
\put(30,180){\small $+$}
\put(30,195){\small $-$}
\put(30,210){\small $+$}
\put(30,225){\small $-$}
\put(30,240){\small $+$}
\put(30,255){\small $-$}
\put(30,270){\small $+$}
}
}
%
% fifth diagram
\put(280,0)
{
\put(0,60){\circle{5}}
\put(0,135){\circle{5}}
\put(0,210){\circle{5}}
\put(30,0){\circle*{5}}
\put(30,15){\circle*{5}}
\put(30,30){\circle*{5}}
\put(30,45){\circle*{5}}
\put(30,60){\circle*{5}}
\put(30,75){\circle*{5}}
\put(30,90){\circle*{5}}
\put(30,105){\circle*{5}}
\put(30,120){\circle*{5}}
\put(30,135){\circle*{5}}
\put(30,150){\circle*{5}}
\put(30,165){\circle*{5}}
\put(30,180){\circle*{5}}
\put(30,195){\circle*{5}}
\put(30,210){\circle*{5}}
\put(30,225){\circle*{5}}
\put(30,240){\circle*{5}}
\put(30,255){\circle*{5}}
\put(30,270){\circle*{5}}
%
% vertical arrows
\put(30,3){\vector(0,1){9}}
\put(30,27){\vector(0,-1){9}}
\put(30,33){\vector(0,1){9}}
\put(30,57){\vector(0,-1){9}}
\put(30,63){\vector(0,1){9}}
\put(30,87){\vector(0,-1){9}}
\put(30,93){\vector(0,1){9}}
\put(30,117){\vector(0,-1){9}}
\put(30,123){\vector(0,1){9}}
\put(30,147){\vector(0,-1){9}}
\put(30,153){\vector(0,1){9}}
\put(30,177){\vector(0,-1){9}}
\put(30,183){\vector(0,1){9}}
\put(30,207){\vector(0,-1){9}}
\put(30,213){\vector(0,1){9}}
\put(30,237){\vector(0,-1){9}}
\put(30,243){\vector(0,1){9}}
\put(30,267){\vector(0,-1){9}}
\put(0,132){\vector(0,-1){69}}
\put(0,138){\vector(0,1){69}}
% diagonal arrows
\put(3,60){\vector(1,0){24}}
\put(27,77){\vector(-1,2){23}}
\put(3,128){\vector(2,-3){24}}
\put(27,107){\vector(-1,1){23}}
\put(3,133){\vector(2,-1){24}}
\put(27,135){\vector(-1,0){23}}
\put(3,137){\vector(2,1){24}}
\put(27,163){\vector(-1,-1){23}}
\put(3,142){\vector(2,3){24}}
\put(27,193){\vector(-1,-2){23}}
\put(3,210){\vector(1,0){24}}
\put(3,-2)
{
\put(-17,60){\small $-$}
\put(-17,135){\small $+$}
\put(-17,210){\small $-$}
}
\put(4,-2)
{
\put(30,0){\small $+$}
\put(30,15){\small $-$}
\put(30,30){\small $+$}
\put(30,45){\small $-$}
\put(30,60){\small $+$}
\put(30,75){\small $-$}
\put(30,90){\small $+$}
\put(30,105){\small $-$}
\put(30,120){\small $+$}
\put(30,135){\small $-$}
\put(30,150){\small $+$}
\put(30,165){\small $-$}
\put(30,180){\small $+$}
\put(30,195){\small $-$}
\put(30,210){\small $+$}
\put(30,225){\small $-$}
\put(30,240){\small $+$}
\put(30,255){\small $-$}
\put(30,270){\small $+$}
}
}
\put(10,-20){$Q_1$}
\put(80,-20){$Q_2$}
\put(150,-20){$Q_3$}
\put(220,-20){$Q_4$}
\put(290,-20){$Q_5$}
\put(-50,133){${\scriptstyle\ell -1}\left\{ \makebox(0,75){}\right.$}
\put(330,133){$\left. \makebox(0,140){}\right\}{\scriptstyle t\ell -1}$}
\end{picture}
%
%%%%%%%%%%%%%%%%%%%%%%%%%%%%%%%%%%%
%
\begin{picture}(360,380)(-60,-15)
%
% first diagram
\put(0,0)
{
\put(0,60){\circle{5}}
\put(0,135){\circle{5}}
\put(0,210){\circle{5}}
\put(0,285){\circle{5}}
\put(30,0){\circle*{5}}
\put(30,15){\circle*{5}}
\put(30,30){\circle*{5}}
\put(30,45){\circle*{5}}
\put(30,60){\circle*{5}}
\put(30,75){\circle*{5}}
\put(30,90){\circle*{5}}
\put(30,105){\circle*{5}}
\put(30,120){\circle*{5}}
\put(30,135){\circle*{5}}
\put(30,150){\circle*{5}}
\put(30,165){\circle*{5}}
\put(30,180){\circle*{5}}
\put(30,195){\circle*{5}}
\put(30,210){\circle*{5}}
\put(30,225){\circle*{5}}
\put(30,240){\circle*{5}}
\put(30,255){\circle*{5}}
\put(30,270){\circle*{5}}
\put(30,285){\circle*{5}}
\put(30,300){\circle*{5}}
\put(30,315){\circle*{5}}
\put(30,330){\circle*{5}}
\put(30,345){\circle*{5}}
%
% vertical arrows
\put(30,3){\vector(0,1){9}}
\put(30,27){\vector(0,-1){9}}
\put(30,33){\vector(0,1){9}}
\put(30,57){\vector(0,-1){9}}
\put(30,63){\vector(0,1){9}}
\put(30,87){\vector(0,-1){9}}
\put(30,93){\vector(0,1){9}}
\put(30,117){\vector(0,-1){9}}
\put(30,123){\vector(0,1){9}}
\put(30,147){\vector(0,-1){9}}
\put(30,153){\vector(0,1){9}}
\put(30,177){\vector(0,-1){9}}
\put(30,183){\vector(0,1){9}}
\put(30,207){\vector(0,-1){9}}
\put(30,213){\vector(0,1){9}}
\put(30,237){\vector(0,-1){9}}
\put(30,243){\vector(0,1){9}}
\put(30,267){\vector(0,-1){9}}
\put(30,273){\vector(0,1){9}}
\put(30,297){\vector(0,-1){9}}
\put(30,303){\vector(0,1){9}}
\put(30,327){\vector(0,-1){9}}
\put(30,333){\vector(0,1){9}}
\put(0,132){\vector(0,-1){69}}
\put(0,138){\vector(0,1){69}}
\put(0,282){\vector(0,-1){69}}
% diagonal arrows
\put(3,51){\vector(1,-2){24}}
\put(27,19){\vector(-2,3){23}}
\put(3,58){\vector(1,-1){24}}
\put(27,47){\vector(-2,1){23}}
\put(3,60){\vector(1,0){24}}
\put(27,73){\vector(-2,-1){23}}
\put(3,62){\vector(1,1){24}}
\put(27,101){\vector(-2,-3){23}}
\put(3,69){\vector(1,2){24}}
\put(27,135){\vector(-1,0){24}}
\put(3,201){\vector(1,-2){24}}
\put(27,169){\vector(-2,3){23}}
\put(3,208){\vector(1,-1){24}}
\put(27,197){\vector(-2,1){23}}
\put(3,210){\vector(1,0){24}}
\put(27,223){\vector(-2,-1){23}}
\put(3,212){\vector(1,1){24}}
\put(27,251){\vector(-2,-3){23}}
\put(3,219){\vector(1,2){24}}
\put(27,285){\vector(-1,0){24}}
\put(3,-2)
{
\put(-17,60){\small $-$}
\put(-17,135){\small $+$}
\put(-17,210){\small $-$}
\put(-17,285){\small $+$}
}
\put(4,-2)
{
\put(30,0){\small $+$}
\put(30,15){\small $-$}
\put(30,30){\small $+$}
\put(30,45){\small $-$}
\put(30,60){\small $+$}
\put(30,75){\small $-$}
\put(30,90){\small $+$}
\put(30,105){\small $-$}
\put(30,120){\small $+$}
\put(30,135){\small $-$}
\put(30,150){\small $+$}
\put(30,165){\small $-$}
\put(30,180){\small $+$}
\put(30,195){\small $-$}
\put(30,210){\small $+$}
\put(30,225){\small $-$}
\put(30,240){\small $+$}
\put(30,255){\small $-$}
\put(30,270){\small $+$}
\put(30,285){\small $-$}
\put(30,300){\small $+$}
\put(30,315){\small $-$}
\put(30,330){\small $+$}
\put(30,345){\small $-$}
}
}
% second diagram
\put(70,0)
{
\put(0,60){\circle{5}}
\put(0,135){\circle{5}}
\put(0,210){\circle{5}}
\put(0,285){\circle{5}}
\put(30,0){\circle*{5}}
\put(30,15){\circle*{5}}
\put(30,30){\circle*{5}}
\put(30,45){\circle*{5}}
\put(30,60){\circle*{5}}
\put(30,75){\circle*{5}}
\put(30,90){\circle*{5}}
\put(30,105){\circle*{5}}
\put(30,120){\circle*{5}}
\put(30,135){\circle*{5}}
\put(30,150){\circle*{5}}
\put(30,165){\circle*{5}}
\put(30,180){\circle*{5}}
\put(30,195){\circle*{5}}
\put(30,210){\circle*{5}}
\put(30,225){\circle*{5}}
\put(30,240){\circle*{5}}
\put(30,255){\circle*{5}}
\put(30,270){\circle*{5}}
\put(30,285){\circle*{5}}
\put(30,300){\circle*{5}}
\put(30,315){\circle*{5}}
\put(30,330){\circle*{5}}
\put(30,345){\circle*{5}}
%
% vertical arrows
\put(30,3){\vector(0,1){9}}
\put(30,27){\vector(0,-1){9}}
\put(30,33){\vector(0,1){9}}
\put(30,57){\vector(0,-1){9}}
\put(30,63){\vector(0,1){9}}
\put(30,87){\vector(0,-1){9}}
\put(30,93){\vector(0,1){9}}
\put(30,117){\vector(0,-1){9}}
\put(30,123){\vector(0,1){9}}
\put(30,147){\vector(0,-1){9}}
\put(30,153){\vector(0,1){9}}
\put(30,177){\vector(0,-1){9}}
\put(30,183){\vector(0,1){9}}
\put(30,207){\vector(0,-1){9}}
\put(30,213){\vector(0,1){9}}
\put(30,237){\vector(0,-1){9}}
\put(30,243){\vector(0,1){9}}
\put(30,267){\vector(0,-1){9}}
\put(30,273){\vector(0,1){9}}
\put(30,297){\vector(0,-1){9}}
\put(30,303){\vector(0,1){9}}
\put(30,327){\vector(0,-1){9}}
\put(30,333){\vector(0,1){9}}
\put(0,63){\vector(0,1){69}}
\put(0,207){\vector(0,-1){69}}
\put(0,213){\vector(0,1){69}}
% diagonal arrows
\put(27,19){\vector(-2,3){23}}
\put(3,58){\vector(1,-1){24}}
\put(27,47){\vector(-2,1){23}}
\put(3,60){\vector(1,0){24}}
\put(27,73){\vector(-2,-1){23}}
\put(3,62){\vector(1,1){24}}
\put(27,101){\vector(-2,-3){23}}
\put(3,133){\vector(2,-1){24}}
\put(27,135){\vector(-1,0){23}}
\put(3,137){\vector(2,1){24}}
\put(27,169){\vector(-2,3){23}}
\put(3,208){\vector(1,-1){24}}
\put(27,197){\vector(-2,1){23}}
\put(3,210){\vector(1,0){24}}
\put(27,223){\vector(-2,-1){23}}
\put(3,212){\vector(1,1){24}}
\put(27,251){\vector(-2,-3){23}}
\put(3,283){\vector(2,-1){24}}
\put(27,285){\vector(-1,0){23}}
\put(3,287){\vector(2,1){24}}
\put(3,-2)
{
\put(-17,60){\small $+$}
\put(-17,135){\small $-$}
\put(-17,210){\small $+$}
\put(-17,285){\small $-$}
}
\put(4,-2)
{
\put(30,0){\small $+$}
\put(30,15){\small $-$}
\put(30,30){\small $+$}
\put(30,45){\small $-$}
\put(30,60){\small $+$}
\put(30,75){\small $-$}
\put(30,90){\small $+$}
\put(30,105){\small $-$}
\put(30,120){\small $+$}
\put(30,135){\small $-$}
\put(30,150){\small $+$}
\put(30,165){\small $-$}
\put(30,180){\small $+$}
\put(30,195){\small $-$}
\put(30,210){\small $+$}
\put(30,225){\small $-$}
\put(30,240){\small $+$}
\put(30,255){\small $-$}
\put(30,270){\small $+$}
\put(30,285){\small $-$}
\put(30,300){\small $+$}
\put(30,315){\small $-$}
\put(30,330){\small $+$}
\put(30,345){\small $-$}
}
}
% third diagram
\put(140,0)
{
\put(0,60){\circle{5}}
\put(0,135){\circle{5}}
\put(0,210){\circle{5}}
\put(0,285){\circle{5}}
\put(30,0){\circle*{5}}
\put(30,15){\circle*{5}}
\put(30,30){\circle*{5}}
\put(30,45){\circle*{5}}
\put(30,60){\circle*{5}}
\put(30,75){\circle*{5}}
\put(30,90){\circle*{5}}
\put(30,105){\circle*{5}}
\put(30,120){\circle*{5}}
\put(30,135){\circle*{5}}
\put(30,150){\circle*{5}}
\put(30,165){\circle*{5}}
\put(30,180){\circle*{5}}
\put(30,195){\circle*{5}}
\put(30,210){\circle*{5}}
\put(30,225){\circle*{5}}
\put(30,240){\circle*{5}}
\put(30,255){\circle*{5}}
\put(30,270){\circle*{5}}
\put(30,285){\circle*{5}}
\put(30,300){\circle*{5}}
\put(30,315){\circle*{5}}
\put(30,330){\circle*{5}}
\put(30,345){\circle*{5}}
%
% vertical arrows
\put(30,3){\vector(0,1){9}}
\put(30,27){\vector(0,-1){9}}
\put(30,33){\vector(0,1){9}}
\put(30,57){\vector(0,-1){9}}
\put(30,63){\vector(0,1){9}}
\put(30,87){\vector(0,-1){9}}
\put(30,93){\vector(0,1){9}}
\put(30,117){\vector(0,-1){9}}
\put(30,123){\vector(0,1){9}}
\put(30,147){\vector(0,-1){9}}
\put(30,153){\vector(0,1){9}}
\put(30,177){\vector(0,-1){9}}
\put(30,183){\vector(0,1){9}}
\put(30,207){\vector(0,-1){9}}
\put(30,213){\vector(0,1){9}}
\put(30,237){\vector(0,-1){9}}
\put(30,243){\vector(0,1){9}}
\put(30,267){\vector(0,-1){9}}
\put(30,273){\vector(0,1){9}}
\put(30,297){\vector(0,-1){9}}
\put(30,303){\vector(0,1){9}}
\put(30,327){\vector(0,-1){9}}
\put(30,333){\vector(0,1){9}}
\put(0,132){\vector(0,-1){69}}
\put(0,138){\vector(0,1){69}}
\put(0,282){\vector(0,-1){69}}
% diagonal arrows
\put(3,58){\vector(1,-1){24}}
\put(27,47){\vector(-2,1){23}}
\put(3,60){\vector(1,0){24}}
\put(27,73){\vector(-2,-1){23}}
\put(3,62){\vector(1,1){24}}
\put(27,107){\vector(-1,1){23}}
\put(3,133){\vector(2,-1){24}}
\put(27,135){\vector(-1,0){23}}
\put(3,137){\vector(2,1){24}}
\put(27,163){\vector(-1,-1){23}}
\put(3,208){\vector(1,-1){24}}
\put(27,197){\vector(-2,1){23}}
\put(3,210){\vector(1,0){24}}
\put(27,223){\vector(-2,-1){23}}
\put(3,212){\vector(1,1){24}}
\put(27,257){\vector(-1,1){23}}
\put(3,283){\vector(2,-1){24}}
\put(27,285){\vector(-1,0){23}}
\put(3,287){\vector(2,1){24}}
\put(27,313){\vector(-1,-1){23}}
\put(3,-2)
{
\put(-17,60){\small $-$}
\put(-17,135){\small $+$}
\put(-17,210){\small $-$}
\put(-17,285){\small $+$}
}
\put(4,-2)
{
\put(30,0){\small $+$}
\put(30,15){\small $-$}
\put(30,30){\small $+$}
\put(30,45){\small $-$}
\put(30,60){\small $+$}
\put(30,75){\small $-$}
\put(30,90){\small $+$}
\put(30,105){\small $-$}
\put(30,120){\small $+$}
\put(30,135){\small $-$}
\put(30,150){\small $+$}
\put(30,165){\small $-$}
\put(30,180){\small $+$}
\put(30,195){\small $-$}
\put(30,210){\small $+$}
\put(30,225){\small $-$}
\put(30,240){\small $+$}
\put(30,255){\small $-$}
\put(30,270){\small $+$}
\put(30,285){\small $-$}
\put(30,300){\small $+$}
\put(30,315){\small $-$}
\put(30,330){\small $+$}
\put(30,345){\small $-$}
}
}
%
% fourth diagram
\put(210,0)
{
\put(0,60){\circle{5}}
\put(0,135){\circle{5}}
\put(0,210){\circle{5}}
\put(0,285){\circle{5}}
\put(30,0){\circle*{5}}
\put(30,15){\circle*{5}}
\put(30,30){\circle*{5}}
\put(30,45){\circle*{5}}
\put(30,60){\circle*{5}}
\put(30,75){\circle*{5}}
\put(30,90){\circle*{5}}
\put(30,105){\circle*{5}}
\put(30,120){\circle*{5}}
\put(30,135){\circle*{5}}
\put(30,150){\circle*{5}}
\put(30,165){\circle*{5}}
\put(30,180){\circle*{5}}
\put(30,195){\circle*{5}}
\put(30,210){\circle*{5}}
\put(30,225){\circle*{5}}
\put(30,240){\circle*{5}}
\put(30,255){\circle*{5}}
\put(30,270){\circle*{5}}
\put(30,285){\circle*{5}}
\put(30,300){\circle*{5}}
\put(30,315){\circle*{5}}
\put(30,330){\circle*{5}}
\put(30,345){\circle*{5}}
%
% vertical arrows
\put(30,3){\vector(0,1){9}}
\put(30,27){\vector(0,-1){9}}
\put(30,33){\vector(0,1){9}}
\put(30,57){\vector(0,-1){9}}
\put(30,63){\vector(0,1){9}}
\put(30,87){\vector(0,-1){9}}
\put(30,93){\vector(0,1){9}}
\put(30,117){\vector(0,-1){9}}
\put(30,123){\vector(0,1){9}}
\put(30,147){\vector(0,-1){9}}
\put(30,153){\vector(0,1){9}}
\put(30,177){\vector(0,-1){9}}
\put(30,183){\vector(0,1){9}}
\put(30,207){\vector(0,-1){9}}
\put(30,213){\vector(0,1){9}}
\put(30,237){\vector(0,-1){9}}
\put(30,243){\vector(0,1){9}}
\put(30,267){\vector(0,-1){9}}
\put(30,273){\vector(0,1){9}}
\put(30,297){\vector(0,-1){9}}
\put(30,303){\vector(0,1){9}}
\put(30,327){\vector(0,-1){9}}
\put(30,333){\vector(0,1){9}}
\put(0,63){\vector(0,1){69}}
\put(0,207){\vector(0,-1){69}}
\put(0,213){\vector(0,1){69}}
% diagonal arrows
\put(27,47){\vector(-2,1){23}}
\put(3,60){\vector(1,0){24}}
\put(27,73){\vector(-2,-1){23}}
\put(3,128){\vector(2,-3){24}}
\put(27,107){\vector(-1,1){23}}
\put(3,133){\vector(2,-1){24}}
\put(27,135){\vector(-1,0){23}}
\put(3,137){\vector(2,1){24}}
\put(3,142){\vector(2,3){23}}
\put(27,163){\vector(-1,-1){24}}
\put(27,197){\vector(-2,1){23}}
\put(3,210){\vector(1,0){24}}
\put(27,223){\vector(-2,-1){23}}
\put(3,278){\vector(2,-3){24}}
\put(27,257){\vector(-1,1){23}}
\put(3,283){\vector(2,-1){24}}
\put(27,285){\vector(-1,0){23}}
\put(3,287){\vector(2,1){24}}
\put(3,292){\vector(2,3){23}}
\put(27,313){\vector(-1,-1){24}}
\put(3,-2)
{
\put(-17,60){\small $+$}
\put(-17,135){\small $-$}
\put(-17,210){\small $+$}
\put(-17,285){\small $-$}
}
\put(4,-2)
{
\put(30,0){\small $+$}
\put(30,15){\small $-$}
\put(30,30){\small $+$}
\put(30,45){\small $-$}
\put(30,60){\small $+$}
\put(30,75){\small $-$}
\put(30,90){\small $+$}
\put(30,105){\small $-$}
\put(30,120){\small $+$}
\put(30,135){\small $-$}
\put(30,150){\small $+$}
\put(30,165){\small $-$}
\put(30,180){\small $+$}
\put(30,195){\small $-$}
\put(30,210){\small $+$}
\put(30,225){\small $-$}
\put(30,240){\small $+$}
\put(30,255){\small $-$}
\put(30,270){\small $+$}
\put(30,285){\small $-$}
\put(30,300){\small $+$}
\put(30,315){\small $-$}
\put(30,330){\small $+$}
\put(30,345){\small $-$}
}
}
%
% fifth diagram
\put(280,0)
{
\put(0,60){\circle{5}}
\put(0,135){\circle{5}}
\put(0,210){\circle{5}}
\put(0,285){\circle{5}}
\put(30,0){\circle*{5}}
\put(30,15){\circle*{5}}
\put(30,30){\circle*{5}}
\put(30,45){\circle*{5}}
\put(30,60){\circle*{5}}
\put(30,75){\circle*{5}}
\put(30,90){\circle*{5}}
\put(30,105){\circle*{5}}
\put(30,120){\circle*{5}}
\put(30,135){\circle*{5}}
\put(30,150){\circle*{5}}
\put(30,165){\circle*{5}}
\put(30,180){\circle*{5}}
\put(30,195){\circle*{5}}
\put(30,210){\circle*{5}}
\put(30,225){\circle*{5}}
\put(30,240){\circle*{5}}
\put(30,255){\circle*{5}}
\put(30,270){\circle*{5}}
\put(30,285){\circle*{5}}
\put(30,300){\circle*{5}}
\put(30,315){\circle*{5}}
\put(30,330){\circle*{5}}
\put(30,345){\circle*{5}}
%
% vertical arrows
\put(30,3){\vector(0,1){9}}
\put(30,27){\vector(0,-1){9}}
\put(30,33){\vector(0,1){9}}
\put(30,57){\vector(0,-1){9}}
\put(30,63){\vector(0,1){9}}
\put(30,87){\vector(0,-1){9}}
\put(30,93){\vector(0,1){9}}
\put(30,117){\vector(0,-1){9}}
\put(30,123){\vector(0,1){9}}
\put(30,147){\vector(0,-1){9}}
\put(30,153){\vector(0,1){9}}
\put(30,177){\vector(0,-1){9}}
\put(30,183){\vector(0,1){9}}
\put(30,207){\vector(0,-1){9}}
\put(30,213){\vector(0,1){9}}
\put(30,237){\vector(0,-1){9}}
\put(30,243){\vector(0,1){9}}
\put(30,267){\vector(0,-1){9}}
\put(30,273){\vector(0,1){9}}
\put(30,297){\vector(0,-1){9}}
\put(30,303){\vector(0,1){9}}
\put(30,327){\vector(0,-1){9}}
\put(30,333){\vector(0,1){9}}
\put(0,132){\vector(0,-1){69}}
\put(0,138){\vector(0,1){69}}
\put(0,282){\vector(0,-1){69}}
% diagonal arrows
\put(3,60){\vector(1,0){24}}
\put(27,77){\vector(-1,2){23}}
\put(3,128){\vector(2,-3){24}}
\put(27,107){\vector(-1,1){23}}
\put(3,133){\vector(2,-1){24}}
\put(27,135){\vector(-1,0){23}}
\put(3,137){\vector(2,1){24}}
\put(27,163){\vector(-1,-1){23}}
\put(3,142){\vector(2,3){24}}
\put(27,193){\vector(-1,-2){23}}
\put(3,210){\vector(1,0){24}}
\put(27,227){\vector(-1,2){23}}
\put(3,278){\vector(2,-3){24}}
\put(27,257){\vector(-1,1){23}}
\put(3,283){\vector(2,-1){24}}
\put(27,285){\vector(-1,0){23}}
\put(3,287){\vector(2,1){24}}
\put(27,313){\vector(-1,-1){23}}
\put(3,292){\vector(2,3){24}}
\put(27,343){\vector(-1,-2){23}}
\put(3,-2)
{
\put(-17,60){\small $-$}
\put(-17,135){\small $+$}
\put(-17,210){\small $-$}
\put(-17,285){\small $+$}
}
\put(4,-2)
{
\put(30,0){\small $+$}
\put(30,15){\small $-$}
\put(30,30){\small $+$}
\put(30,45){\small $-$}
\put(30,60){\small $+$}
\put(30,75){\small $-$}
\put(30,90){\small $+$}
\put(30,105){\small $-$}
\put(30,120){\small $+$}
\put(30,135){\small $-$}
\put(30,150){\small $+$}
\put(30,165){\small $-$}
\put(30,180){\small $+$}
\put(30,195){\small $-$}
\put(30,210){\small $+$}
\put(30,225){\small $-$}
\put(30,240){\small $+$}
\put(30,255){\small $-$}
\put(30,270){\small $+$}
\put(30,285){\small $-$}
\put(30,300){\small $+$}
\put(30,315){\small $-$}
\put(30,330){\small $+$}
\put(30,345){\small $-$}
}
}
\put(10,-20){$Q_1$}
\put(80,-20){$Q_2$}
\put(150,-20){$Q_3$}
\put(220,-20){$Q_4$}
\put(290,-20){$Q_5$}
\put(-50,170.5){${\scriptstyle\ell -1}\left\{ \makebox(0,115){}\right.$}
\put(330,170.5){$\left. \makebox(0,175){}\right\}{\scriptstyle t\ell -1}$}
\end{picture}
\caption{The quiver $Q_{\ell}(M_t)$ with $t=5$  for even $\ell$
(upper) and for odd $\ell$ (lower),
where we identify the right columns in all the quivers
$Q_1$, \dots, $Q_5$.}
\label{fig:quiverG}
\end{figure}

With the Cartan matrix  $M_t$ and $\ell\geq 2$ we associate
 a quiver $Q_{\ell}(M_t)$ as below.
First, as a rather general example, the case $t=5$ is given in
 Fig.~\ref{fig:quiverG},
where the right columns in  the five quivers $Q_1$,\dots,$Q_5$ are
identified.
Also we assign the empty or filled circle $\circ$/$\bullet$
and the sign +/$-$ to  each vertex as shown.
For a general odd $t$,  the quiver $Q_{\ell}(M_t)$ is defined
by naturally extending the case $t=5$.
Namely, we consider $t$ quivers $Q_1$,\dots,$Q_t$.
In each quiver $Q_i$ there are
$\ell-1$ vertices (with $\circ$) in the left column and
$t\ell-1$ vertices (with $\bullet$) in the right column.
The arrows are put as clearly indicated by the example
in Fig.~\ref{fig:quiverG}.
The right columns in all the quivers $Q_1$, \dots, $Q_t$
 are identified.

Let us choose  the index set $\mathbf{I}$
of the vertices of $Q_{\ell}(M_t)$
so that $\mathbf{i}=(i,i')\in \mathbf{I}$ represents
the vertex 
at the $i'$th row (from the bottom)
of the left column in $Q_i$ for $i=1,\dots,t$,
 and the one of the right column in any quiver
for $i=t+1$.
Thus, $i=1,\dots,t+1$, and $i'=1,\dots,\ell-1$ if $i\neq t+1$
and $i'=1,\dots,t\ell-1$ if $i=t+1$.

For $k\in \{1,\dots,t\}$,
let $\mathbf{I}^{\circ}_{+,k}$
(resp.\ $\mathbf{I}^{\circ}_{-,k}$)
denote the set of the vertices $\mathbf{i}$ in $Q_k$  with
property $\circ$ and $+$ (resp.\ $\circ$ and $-$).
Similarly, let $\mathbf{I}^{\bullet}_+$
(resp.\ $\mathbf{I}^{\bullet}_-$)
denote the set of the vertices $\mathbf{i}$ with
property $\bullet$ and $+$ (resp.\ $\bullet$ and $-$).
We define composite mutations,
\begin{align}
\label{eq:mupm2}
\mu^{\circ}_{+,k}=\prod_{\mathbf{i}\in\mathbf{I}^{\circ}_{+,k}}
\mu_{\mathbf{i}},
\quad
\mu^{\circ}_{-,k}=\prod_{\mathbf{i}\in\mathbf{I}^{\circ}_{-,k}}
\mu_{\mathbf{i}},
\quad
\mu^{\bullet}_+=\prod_{\mathbf{i}\in\mathbf{I}^{\bullet}_+}
\mu_{\mathbf{i}},
\quad
\mu^{\bullet}_-=\prod_{\mathbf{i}\in\mathbf{I}^{\bullet}_-}
\mu_{\mathbf{i}}.
\end{align}
Note that they do not depend on the order of the product.

For a permutation $w$ of $\{1,\dots,t \}$,
let $\tilde{w}$ be the permutation of
$\mathbf{I}$ such that
$\tilde{w}(i,i')=(w(i),i')$ for $i\neq t+1$
and $(t+1,i')$ for $i=t+1$.
Let $\tilde{w}(Q_{\ell}(M_t))$
denote the
quiver induced from $Q_{\ell}(M_t)$
 by
$\tilde{w}$.
Namely, if there is an arrow
 $\mathbf{i}\rightarrow
\mathbf{j}$ in $Q_{\ell}(M_t)$,
then, there is an arrow
$\tilde{w}(\mathbf{i})
\rightarrow
\tilde{w}(\mathbf{j})
$
in  $\tilde{w}(Q_{\ell}(M_t))$.
For a quiver $Q$,
let $Q^{\mathrm{op}}$ denote the opposite quiver.

\begin{Lemma}
\label{lem:GQmut}
Let $Q(0):=Q_{\ell}(M_t)$.
We have the following periodic sequence of mutations of quivers:
\begin{align}
\label{eq:GB2}
\begin{matrix}
% 1st
Q(0)
& 
\displaystyle
\mathop{\longleftrightarrow}^{
\mu^{\bullet}_+
\mu^{\circ}_{+,1}}
& Q(\frac{1}{t})
&
\displaystyle
\mathop{\longleftrightarrow}^{\mu^{\bullet}_-
\mu^{\circ}_{+,t-1}}
&
Q(\frac{2}{t})
&
\displaystyle
\mathop{\longleftrightarrow}^{\mu^{\bullet}_+
\mu^{\circ}_{+,3}}
&
Q(\frac{3}{t})
&
\displaystyle
\mathop{\longleftrightarrow}^{\mu^{\bullet}_-
\mu^{\circ}_{+,t-3}}
&
Q(\frac{4}{t})
\\
% 2nd
&
\displaystyle
\mathop{\longleftrightarrow}^{\mu^{\bullet}_+
\mu^{\circ}_{+,5}}
&
&
%\boldsymbol{\nu}_{t-2}(Q)
\cdots
&&
\displaystyle
\mathop{\longleftrightarrow}^{\mu^{\bullet}_-
\mu^{\circ}_{+,2}}
&
Q(\frac{t-1}{t})
&
\displaystyle
\mathop{\longleftrightarrow}^{\mu^{\bullet}_+
\mu^{\circ}_{+,t}}
&
Q(1)
&
\\
% 3rd
& 
\displaystyle
\mathop{\longleftrightarrow}^{\mu^{\bullet}_-
\mu^{\circ}_{-,1}}
&
Q(\frac{t+1}{t})
&
\displaystyle
\mathop{\longleftrightarrow}^{\mu^{\bullet}_+
\mu^{\circ}_{-,t-1}}
&
Q(\frac{t+2}{t})
&
\displaystyle
\mathop{\longleftrightarrow}^{\mu^{\bullet}_-
\mu^{\circ}_{-,3}}
&
Q(\frac{t+3}{t})
&
\displaystyle
\mathop{\longleftrightarrow}^{\mu^{\bullet}_+
\mu^{\circ}_{-,t-3}}
&
Q(\frac{t+4}{t})
\\
% 4th
&
\displaystyle
\mathop{\longleftrightarrow}^{\mu^{\bullet}_-
\mu^{\circ}_{-,5}}
&
&
\cdots
&
&
\displaystyle
\mathop{\longleftrightarrow}^{\mu^{\bullet}_+
\mu^{\circ}_{-,2}}
&
Q(\frac{2t-1}{t})
&
\displaystyle
\mathop{\longleftrightarrow}^{\mu^{\bullet}_-
\mu^{\circ}_{-,t}}
&
Q(2)=Q(0).
&
\\
\end{matrix}
\end{align}
Here, the quiver $Q(p/t)$ $(p=1,\dots,2t)$ is defined by
\begin{align}
\label{eq:Qpt}
Q(p/t):=
\begin{cases}
\tilde{w}_p(Q(0))^{\mathrm{op}}&\mbox{$p$: odd}\\
\tilde{w}_p(Q(0))&\mbox{$p$: even},
\end{cases}
\end{align}
and $w_p$ is a permutation
of $\{1,\dots,t\}$ defined by
\begin{align}
w_p&=
\begin{cases}
 r_+ r_- \cdots r_+ \ \mbox{\em{($p$ terms)}} &\mbox{$p$: odd}\\
 r_+ r_- \cdots r_- \ \mbox{\em{($p$ terms)}}& \mbox{$p$: even},
\end{cases}\\
r_+&=(23)(45)\cdots (r-1,r),\quad
r_-=(12)(34)\cdots (r-2,r-1),
\end{align}
where $(ij)$ is the transposition of $i$ and $j$.
\end{Lemma}

\begin{proof}
Let $Q_1$, \dots, $Q_{t}$ be the subquivers in
the definition of 
$Q_{\ell}(M_t)$ as  in Fig.~\ref{fig:quiverG}.
By the sequence of mutations \eqref{eq:GB2},
one can easily check that
$Q_1$ mutates as
\begin{align}
Q_1 \leftrightarrow Q_1^{\mathrm{op}}\leftrightarrow
Q_2 \leftrightarrow Q_3^{\mathrm{op}}\leftrightarrow
\cdots
Q_t^{\mathrm{op}} \leftrightarrow Q_t\leftrightarrow
\cdots
Q_3 \leftrightarrow Q_2^{\mathrm{op}}\leftrightarrow
Q_1,
\end{align}
$Q_2$ mutates as
\begin{align}
Q_2 \leftrightarrow Q_3^{\mathrm{op}}\leftrightarrow
\cdots
\leftrightarrow
Q_t^{\mathrm{op}} \leftrightarrow Q_t\leftrightarrow
\cdots
\leftrightarrow Q_2^{\mathrm{op}}\leftrightarrow
Q_1 \leftrightarrow Q_1^{\mathrm{op}}\leftrightarrow
Q_2,
\end{align}
$Q_3$ mutates as
\begin{align}
Q_3 \leftrightarrow Q_2^{\mathrm{op}}\leftrightarrow
Q_1 \leftrightarrow Q_1^{\mathrm{op}}\leftrightarrow
%Q_2 \leftrightarrow
\cdots
\leftrightarrow
Q_t^{\mathrm{op}} \leftrightarrow Q_t\leftrightarrow
\cdots
\leftrightarrow
 Q_4^{\mathrm{op}}\leftrightarrow Q_3,
\end{align}
and so on.
The result is summarized as \eqref{eq:Qpt}.
\end{proof}

\begin{Example}
The mutation sequence
\eqref{eq:GB2} for $t=5$ is explicitly given 
in Figs. \ref{fig:labelxG1} and \ref{fig:labelxG2},
where only a part of each quiver is presented.
(Caution: the mutations of the top and bottom arrows
may look erroneous but they are correct because of the
effect from the omitted part.)
The encircled vertices are the mutation
points of \eqref{eq:GB2} in the forward direction.
\end{Example}

\begin{figure}[t]
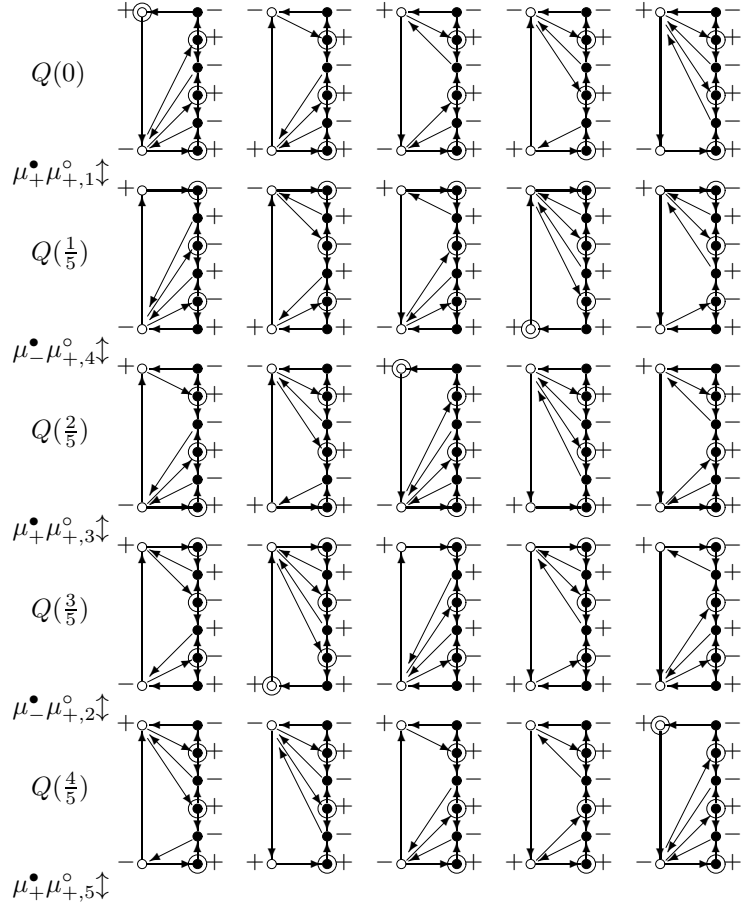

\setlength{\unitlength}{0.7pt}

\caption{(Continues to Fig. \ref{fig:labelxG2}.)
The mutation sequence of the quiver $Q_{\ell}(M_t)$
in  \eqref{eq:GB2} for $t=5$.
Only a part of each quiver is presented.
The encircle vertices correspond to
the mutation points in the forward direction.
}
\label{fig:labelxG1}
\end{figure}
%%%%%%%%%%%%%%%%%%%%%%%%%%%%%%%%%%%%%%%
\begin{figure}[t]
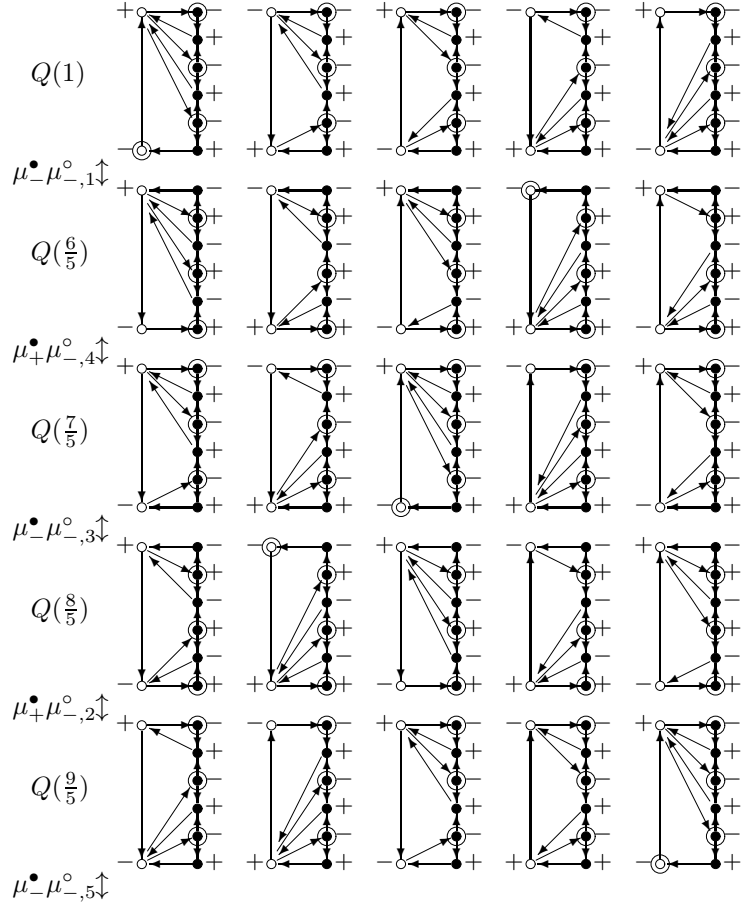

\setlength{\unitlength}{0.7pt}
%
%
%%%%%%%%%%%%%%%%%%%%%%%%%%%%%%%%%%%

%
\caption{(Continues from Fig. \ref{fig:labelxG1}.)
}
\label{fig:labelxG2}
\end{figure}

\subsection{Embedding maps}
Let
 $B=B_{\ell}(M_t)$
be the skew-symmetric matrix corresponding 
to the quiver $Q_{\ell}(M_t)$.
Let $\mathcal{A}(B,x,y)$ 
be the cluster algebra
 with coefficients
in the universal
semifield
$\mathbb{Q}_{\mathrm{sf}}(y)$,
and let
$\mathcal{G}(B,y)$
be
the coefficient group
associated with $\mathcal{A}(B,x,y)$
 as in Section
\ref{sec:cluster}.

In view of Lemma \ref{lem:GQmut}
we set $x(0)=x$, $y(0)=y$ and define 
clusters $x(u)=(x_{\mathbf{i}}(u))_{\mathbf{i}\in \mathbf{I}}$
 ($u\in \frac{1}{t}\mathbb{Z}$)
 and coefficient tuples $y(u)=(y_\mathbf{i}(u))_{\mathbf{i}\in \mathbf{I}}$
 ($u\in \frac{1}{t}\mathbb{Z}$)
by the sequence of mutations
\begin{align}
\label{eq:Gmutseq}
\begin{matrix}
% 1st
\cdots
&
\displaystyle
\mathop{\longleftrightarrow}^{\mu^{\bullet}_-
\mu^{\circ}_{-,t}}
&
(B(0),x(0),y(0))
& 
\displaystyle
\mathop{\longleftrightarrow}^{\mu^{\bullet}_+
\mu^{\circ}_{+,1}}
& (B(\frac{1}{t}),x(\frac{1}{t}),y(\frac{1}{t}))\\
&
\displaystyle
\mathop{\longleftrightarrow}^{\mu^{\bullet}_-
\mu^{\circ}_{+,t-1}}
&
\cdots
&
\displaystyle
\mathop{\longleftrightarrow}^{\mu^{\bullet}_-
\mu^{\circ}_{-,t}}
&
(B(2),x(2),y(2))
&
\displaystyle
\mathop{\longleftrightarrow}^{\mu^{\bullet}_+
\mu^{\circ}_{+,1}}
&
\cdots,
\\
\end{matrix}
\end{align}
where $B(u)$ is the skew-symmetric matrix corresponding to
$Q(u)$.

For  $(\mathbf{i},u)\in
 \mathbf{I}\times \frac{1}{t}\mathbb{Z}$,
we set the parity condition $\mathbf{p}_+$ by
\begin{align}
\label{eq:GQparity}
\mathbf{p}_+:&
\begin{cases}
 \mathbf{i}\in 
\mathbf{I}^{\bullet}_+ \sqcup 
\mathbf{I}^{\circ}_{+,p+1}
& u\equiv \frac{p}{t}, 0\leq p\leq t-1, \mbox{$p$: even}\\
 \mathbf{i}\in 
\mathbf{I}^{\bullet}_- \sqcup 
\mathbf{I}^{\circ}_{+,t-p}
& u\equiv \frac{p}{t}, 0\leq p\leq t-1, \mbox{$p$: odd}\\
 \mathbf{i}\in 
\mathbf{I}^{\bullet}_+ \sqcup 
\mathbf{I}^{\circ}_{-,2t-p}
& u\equiv \frac{p}{t}, t\leq p\leq 2t-1, \mbox{$p$: even}\\
 \mathbf{i}\in 
\mathbf{I}^{\bullet}_- \sqcup 
\mathbf{I}^{\circ}_{-,p+1-t}
& u\equiv \frac{p}{t}, t\leq p\leq 2t-1, \mbox{$p$: odd},\\
\end{cases}
\end{align}
where $\equiv$ is modulo $2\mathbb{Z}$.
We define the condition $\mathbf{p}_-$ by
$(\mathbf{i},u):\mathbf{p}_-   \Longleftrightarrow 
(\mathbf{i},u-1/t):\mathbf{p}_+$.
Plainly speaking,
each $(\mathbf{i},u):\mathbf{p}_+$
(resp. $\mathbf{p}_-$)
is a mutation point of \eqref{eq:Gmutseq} in the forward
(resp. backward) direction of $u$.

There is a correspondence between the parity condition
$\mathbf{p}_{+}$ here and $\mathbf{P}_{+}$,
$\mathbf{P}'_{+}$
in \eqref{eq:GPcond}.

\begin{Lemma}
\label{lem:gmap}
Below $\equiv$ means the equivalence modulo $2\mathbb{Z}$.
\par
(i)
The map
$g: \mathcal{I}_{\ell+}\rightarrow 
\{ (\mathbf{i},u): \mathbf{p}_+
\}
$
\begin{align}
\textstyle
(a,m,u-\frac{d_a}{t})\mapsto 
\begin{cases}
((2j+1,m),u)& a= 1;
m+u\equiv \frac{2j}{t}\\
&(j=0,1,\dots,(t-1)/2)\\
((2t-2j,m),u)& a= 1;
m+u\equiv \frac{2j}{t}\\
&(j=(t+1)/2,\dots,t-1)\\
((t+1,m),u)& \mbox{\rm  $a=2$}\\
\end{cases}
\end{align}
is a bijection.

\par
(ii)
The map
$g': \mathcal{I}'_{\ell+}\rightarrow 
\{ (\mathbf{i},u): \mathbf{p}_+
\}$
\begin{align}
\label{eq:g'}
(a,m,u)\mapsto 
\begin{cases}
((2j+1,m),u)& a= 1;
m+u\equiv \frac{2j}{t}\\
&(j=0,1,\dots,(t-1)/2)\\
((2t-2j,m),u)& a= 1;
m+u\equiv \frac{2j}{t}\\
&(j=(t+1)/2,\dots,t-1)\\
((t+1,m),u)& \mbox{\rm  $a=2$}\\
\end{cases}
\end{align}
is a bijection.
\end{Lemma}
\begin{proof}
The fact (i) is equivalent to (ii) due to
\eqref{eq:PP'}.
So, it is enough to prove  (ii).
Let us examine
the meaning of the  map \eqref{eq:g'} in
 the case $t=5$ with
Fig.~\ref{fig:labelxG1}.
Each encircled vertex therein corresponds to
$(\mathbf{i},u):\mathbf{p}_+$,
and some $(a,m,u)$ is attached to it by $g'$.
For example, in $Q(0)$,
$(1,m,0)$ ($m$: even) are attached to the
vertices with $(\circ,+)$ in the first quiver
(from the left),
and $(2,m,0)$ ($m$: odd) are attached to
the vertices with $(\bullet,+)$.
Similarly,
in $Q(1/5)$,
$(1,m,1/5)$ ($m$: odd) are attached to the
vertices with $(\circ,+)$ in the fourth quiver
(from the left),
and $(2,m,1/5)$ ($m$: even) are attached to
the vertices with $(\bullet,-)$.
Then, one can easily confirm that
$g'$ is indeed a bijection.
A general case is verified similarly.
\end{proof}

We introduce alternative labels
$x_{\mathbf{i}}(u)=x^{(a)}_m(u-d_a/t)$
($(a,m,u-d_a/t)\in \mathcal{I}_{\ell+}$)
for $(\mathbf{i},u)=g((a,m,u-d_a/t))$
and
$y_{\mathbf{i}}(u)=y^{(a)}_m(u)$
($(a,m,u)\in \mathcal{I}'_{\ell+}$)
for $(\mathbf{i},u)=g'((a,m,u))$,
respectively.

\begin{Remark} In the  case $t=1$,
i.e., the simply laced case,
the map $g$ in Lemma \ref{lem:gmap}
reads $(a,m,u)\mapsto ((a,m),u+1)$,
thus, differs from the  simpler one
$(a,m,u)\mapsto ((a,m),u)$ used in
Refs.~\refcite{Inoue09} and \refcite{Kuniba09}.
Either will serve as a natural parametrization
and the transferring from one to the other is easy.
\end{Remark}

\subsection{T-system and cluster algebra}

We show that the T-system $\mathbb{T}_{\ell}(M_t)$
naturally appears as a system of relations
among the cluster variables
$x_{\mathbf{i}}(u)$ 
in the trivial evaluation of coefficients.
(The quiver $Q_{\ell}(M_t)$ is designed to do so.)
Let $\mathcal{A}(B,x)$ be the cluster algebra
with trivial coefficients, where $(B,x)$ is
the initial seed.
Let $\mathbf{1}=\{1\}$ 
be the {\em trivial semifield}
and $\pi_{\mathbf{1}}:
\mathbb{Q}_{\mathrm{sf}}(y)\rightarrow 
\mathbf{1}$, $y_{\mathbf{i}}\mapsto 1$ be the projection.
Let $[x_{\mathbf{i}}(u)]_{\mathbf{1}}$
denote the image of $x_{\mathbf{i}}(u)$
 by the algebra homomorphism
$\mathcal{A}(B,x,y)\rightarrow \mathcal{A}(B,x)$
 induced from $\pi_{\mathbf{1}}$.
It is called the {\em trivial evaluation}.

%Let $\mathbf{I}^{\bullet}$ (resp. $\mathbf{I}^{\circ}$)
%be the set of the vertices $\mathbf{i}\in \mathbf{I}$
%with property $\bullet$ (resp. $\circ$).

\begin{Lemma}
\label{lem:Gx2}
Let $G(b,k,v;a,m,u)$ be the one in \eqref{eq:Tu}.
The family $\{x^{(a)}_m(u)
\mid (a,m,u)\in \mathcal{I}_{\ell+}\}$
satisfies a system of relations
\begin{align}
\label{eq:xy1}
\begin{split}
x^{(a)}_{m}\left(u-\textstyle\frac{d_a}{t}\right)
x^{(a)}_{m}\left(u+\textstyle\frac{d_a}{t}\right)
&=
%\frac{y^{(a)}_m(u+\frac{t-1}{t})}{1+y^{(a)}_m(u+\frac{t-1}{t})}
\frac{y^{(a)}_m(u)}{1+y^{(a)}_m(u)}
\prod_{(b,k,v)\in \mathcal{I}_{\ell+}}
x^{(b)}_{k}(v)^{G(b,k,v;\,a,m,u)}\\
& \quad +
\frac{1}{1+y^{(a)}_m(u)}
x^{(a)}_{m-1}(u)x^{(a)}_{m+1}(u),
\end{split}
\end{align}
where $(a,m,u)\in \mathcal{I}'_{\ell+}$.
In particular,
the family $\{ [x^{(a)}_m(u)]_{\mathbf{1}}
\mid (a,m,u)\in \mathcal{I}_{\ell+}\}$
satisfies the T-system $\mathbb{T}_{\ell}(M_t)$
in $\mathcal{A}(B,x)$
by replacing $T^{(a)}_m(u)$ with $[x^{(a)}_m(u)]_{\mathbf{1}}$.
\end{Lemma}
\begin{proof}
This follows from the exchange relation of
cluster variables \eqref{eq:clust}
and the property of the sequence \eqref{eq:GB2}
which are observed in Figs. \ref{fig:labelxG1} and \ref{fig:labelxG2}.

Let us demonstrate how to obtain these relations
in the case $t=5$ using Figs.~\ref{fig:labelxG1} and \ref{fig:labelxG2}.
For example, consider the mutation at
$((1,2),0)$.
Then, the attached variable
$x^{(1)}_2(-1)$ is mutated to
\begin{align}
\frac{1}{x^{(1)}_2(-1)}
\left\{
\frac{y^{(1)}_2(0)}{1+y^{(1)}_2(0)}
x^{(2)}_{10}(0)
+
\frac{1}{1+y^{(1)}_2(0)}
x^{(1)}_1(0)x^{(1)}_3(0)
\right\},
\end{align}
which should equal to $x^{(1)}_2(1)$. 
Also, consider the mutation at, say,
$((2,9),0)$.
Then, the attached variable
$x^{(2)}_9(-1/5)$ is mutated to
\begin{align}
\frac{1}{x^{(2)}_9(-\frac{1}{5})}
\Biggl\{
&
\frac{y^{(2)}_9(0)}{1+y^{(2)}_9(0)}
x^{(1)}_{1}(0)
\textstyle
x^{(1)}_{2}(-\frac{3}{5})
x^{(1)}_{2}(-\frac{1}{5})
x^{(1)}_{2}(\frac{1}{5}
)x^{(1)}_{2}(\frac{3}{5})
\\
&
+
\frac{1}{1+y^{(2)}_9(0)}
x^{(2)}_{8}(0)x^{(2)}_{10}(0)
\Biggr\},
\end{align}
which should equal to $x^{(2)}_9(1/5)$. 
They certainly agree with 
\eqref{eq:T1} and \eqref{eq:T2}.
(The quiver $Q_{\ell}(M_t)$ is designed to do so.)
\end{proof}

\begin{Definition}
\label{def:T-sub}
The {\em T-subalgebra
$\mathcal{A}_T(B,x)$
of ${\mathcal{A}}(B,x)$
associated with the sequence \eqref{eq:Gmutseq}}
is the subring of
${\mathcal{A}}(B,x)$
generated by
$[x_{\mathbf{i}}(u)]_{\mathbf{1}}$
($(\mathbf{i},u)\in \mathbf{I}\times \frac{1}{t}\mathbb{Z}$).
\end{Definition}

By Lemma \ref{lem:Gx2}, we have the following embedding.
\begin{Theorem}
\label{thm:GTiso}
The ring $\EuScript{T}^{\circ}_{\ell}(M_t)_+$ is isomorphic to
$\mathcal{A}_T(B,x)$ by the correspondence
$T^{(a)}_m(u)\mapsto [x^{(a)}_m(u)]_{\mathbf{1}}$.
\end{Theorem}
\begin{proof}
The map $\rho :T^{(a)}_m(u)\mapsto [x^{(a)}_m(u)]_{\mathbf{1}}$
is a ring homomorphism due to Lemma \ref{lem:Gx2}.
We can construct the inverse of $\rho$ as follows.
For each $\mathbf{i}\in 
\mathbf{I}$, let $u_{\mathbf{i}}
\in \frac{1}{t}\mathbb{Z}$ be the smallest
nonnegative $u_{\mathbf{i}}$ such that $(\mathbf{i},u):\mathbf{p}_+$.
Then, thanks to Lemma \ref{lem:gmap} (i)
there is a unique $(a,m,u_{\mathbf{i}}-d_a/t)\in \mathcal{I}_{\ell+}$
such that $g((a,m,u_{\mathbf{i}}-d_a/t))
=(\mathbf{i},u_{\mathbf{i}})$.
We define a ring homomorphism
$\tilde{\varphi}:\mathbb{Z}[x_{\mathbf{i}}^{\pm1}]_{\mathbf{i}\in
\mathbf{I}}
\rightarrow \EuScript{T}_{\ell}(M_t)$
by
$x_{\mathbf{i}}^{\pm1}\mapsto T^{(a)}_m(u_{\mathbf{i}}-d_a/t)^{\pm1}$.
Thus, we have $\tilde{\varphi}:
[x^{(a)}_m(u_{\mathbf{i}}-d_a/t)]_{\mathbf{1}}
\mapsto  T^{(a)}_m(u_{\mathbf{i}}-d_a/t)$.
Furthermore, one can prove that 
$\tilde{\varphi}:
[x^{(a)}_m(u)]_{\mathbf{1}}
\mapsto  T^{(a)}_m(u)$
for any $(a,m,u)\in \mathcal{I}_{\ell+}$
by induction on the forward and backward mutations,
applying the same T-systems for the both sides.
By the the restriction of $\tilde{\varphi}$
to $\mathcal{A}_T(B,x)$,
we obtain a ring homomorphism
$\varphi:
\mathcal{A}_T(B,x)\rightarrow 
\EuScript{T}^{\circ}_{\ell}(M_t)_+$,
which is the inverse of $\rho$.
\end{proof}

\subsection{Y-system and cluster algebra}

The Y-system $\mathbb{Y}_{\ell}(M_t)$
also naturally appears as a system of relations
among the coefficients
$y_{\mathbf{i}}(u)$.

The following lemma follows from the exchange relation of
coefficients and the property of the sequence
\eqref{eq:GB2}.

\begin{Lemma}
\label{lem:Gy2}
The family $\{ y^{(a)}_m(u)
\mid (a,m,u)\in \mathcal{I}'_{\ell+}\}$
satisfies the Y-system $\mathbb{Y}_{\ell}(M_t)$
by replacing $Y^{(a)}_m(u)$ with $y^{(a)}_m(u)$.
\end{Lemma}
\begin{proof}
Again, let us demonstrate how to obtain these relations
in the case $t=5$ using Figs.~\ref{fig:labelxG1} and \ref{fig:labelxG2}.
For example, consider the mutation at
$((1,2),0)$.
Then, the attached variable
$y^{(1)}_2(0)$ is mutated to
$y^{(1)}_2(0)^{-1}$.
Then, the following factors are multiplied
to $y^{(1)}_2(0)^{-1}$ during $u=\frac{1}{5},\dots,\frac{9}{5}$:
\begin{align*}
&
\textstyle(1+y^{(2)}_{14}(\frac{5}{5})),\\
&\textstyle
(1+y^{(2)}_{13}(\frac{4}{5}))
(1+y^{(2)}_{13}(\frac{6}{5})),\\
&\textstyle
(1+y^{(2)}_{12}(\frac{3}{5}))
(1+y^{(2)}_{12}(\frac{5}{5}))
(1+y^{(2)}_{12}(\frac{7}{5})),\\
&\textstyle
(1+y^{(2)}_{11}(\frac{2}{5}))
(1+y^{(2)}_{11}(\frac{4}{5}))
(1+y^{(2)}_{11}(\frac{6}{5}))
(1+y^{(2)}_{11}(\frac{8}{5})),\\
&\textstyle
(1+y^{(2)}_{10}(\frac{1}{5}))
(1+y^{(2)}_{10}(\frac{3}{5}))
(1+y^{(2)}_{10}(\frac{5}{5}))
(1+y^{(2)}_{10}(\frac{7}{5}))
(1+y^{(2)}_{10}(\frac{9}{5})),\\
&\textstyle
(1+y^{(2)}_{9}(\frac{2}{5}))
(1+y^{(2)}_{9}(\frac{4}{5}))
(1+y^{(2)}_{9}(\frac{6}{5}))
(1+y^{(2)}_{9}(\frac{8}{5})),\\
&\textstyle
(1+y^{(2)}_{8}(\frac{3}{5}))
(1+y^{(2)}_{8}(\frac{5}{5}))
(1+y^{(2)}_{8}(\frac{7}{5})),\\
&\textstyle
(1+y^{(2)}_{7}(\frac{4}{5}))
(1+y^{(2)}_{7}(\frac{6}{5})),\\
&\textstyle
(1+y^{(2)}_{8}(\frac{5}{5})),\\
&(1+y^{(1)}_{1}(1)^{-1})^{-1}(1+y^{(1)}_{3}(1)^{-1})^{-1}.
\end{align*}
The result should equal to $y^{(1)}_2(2)$. 
Also, consider the mutation at, say,
$((2,9),0)$.
Then, the attached variable
$y^{(2)}_9(0)$ is mutated to
$y^{(2)}_9(0)^{-1}$.
Then, the following factors are multiplied
to $y^{(2)}_9(0)^{-1}$ at $u=\frac{1}{5}$:
\begin{align*}
\textstyle
(1+y^{(2)}_{8}(\frac{1}{5})^{-1})^{-1}
(1+y^{(2)}_{10}(\frac{1}{5})^{-1})^{-1}.
\end{align*}
The result should equal to $y^{(2)}_9(\frac{2}{5})$. 
They certainly agree with 
\eqref{eq:Y1} and \eqref{eq:Y2}.

\end{proof}

\begin{Definition}
\label{def:Y-sub}
The {\em Y-subgroup
$\mathcal{G}_Y(B,y)$
of ${\mathcal{G}}(B,y)$
associated with the sequence \eqref{eq:Gmutseq}}
is the subgroup of
${\mathcal{G}}(B,y)$ 
generated by
$y_{\mathbf{i}}(u)$
($(\mathbf{i},u)\in \mathbf{I}\times \frac{1}{t}\mathbb{Z}$)
and $1+y_{\mathbf{i}}(u)$
($(\mathbf{i},u):\mathbf{p}_+$ or $\mathbf{p}_-$).
\end{Definition}

By Lemma \ref{lem:Gy2}, we have the following embedding.
\begin{Theorem}
\label{thm:GYiso}
The group $\EuScript{Y}^{\circ}_{\ell}(M_t)_+$ is isomorphic to
$\mathcal{G}_Y(B,y)$ by the correspondence
$Y^{(a)}_m(u)\mapsto y^{(a)}_m(u)$
and $1+Y^{(a)}_m(u)\mapsto 1+y^{(a)}_m(u)$.
\end{Theorem}
\begin{proof}
%This is parallel to Theorem \ref{thm:GTiso}.
The map $\rho :Y^{(a)}_m(u)\mapsto y^{(a)}_m(u)$,
$1+Y^{(a)}_m(u)\mapsto 1+ y^{(a)}_m(u)$
is a group homomorphism due to Lemma \ref{lem:Gx2}.
We can construct the inverse of $\rho$ as follows.
For each $\mathbf{i}\in 
\mathbf{I}$, let $u_{\mathbf{i}}
\in \frac{1}{t}\mathbb{Z}$ be the largest
nonpositive $u_{\mathbf{i}}$ such that $(\mathbf{i},u_{\mathbf{i}})
:\mathbf{p}_+$.
Then, thanks to Lemma \ref{lem:gmap} (ii)
there is a unique $(a,m,u_{\mathbf{i}})\in \mathcal{I}'_{\ell+}$
such that $g'((a,m,u_{\mathbf{i}}))=(\mathbf{i},u_{\mathbf{i}})$.
We define a semifield homomorphism
$\tilde{\varphi}:\mathbb{Q}_{\mathrm{sf}}(y_{\mathbf{i}})_{\mathbf{i}\in
\mathbf{I}}
\rightarrow \EuScript{Y}_{\ell}(M_t)$
as follows.
If $u_{\mathbf{i}}=0$,
then
$y_{\mathbf{i}}\mapsto Y^{(a)}_m(0)$.
If $u_{\mathbf{i}}<0$, we define
\begin{align}
\tilde{\varphi}(y_{\mathbf{i}})
=
Y^{(a)}_{m}(u_{\mathbf{i}})^{-1}
\frac{
\displaystyle
\prod_{(b,k,v)}(1+Y^{(b)}_k (v))
}
{
\displaystyle
\prod_{(b,k,v)}(1+Y^{(b)}_k (v)^{-1})
},
\end{align}
where the product in the numerator is taken
for $(b,k,v):\mathcal{I}'_{\ell+}$
such that $u_{\mathbf{i}}< v < 0$ and
$B_{\mathbf{j}\mathbf{i}}(v)<0$ for 
$(\mathbf{j},v)=g'((b,k,v))$,
and the product in the denominator is taken
for $(b,k,v):\mathcal{I}'_{\ell+}$
such that $u_{\mathbf{i}}< v < 0$ and
$B_{\mathbf{j}\mathbf{i}}(v)>0$ for 
$(\mathbf{j},v)=g'((b,k,v))$.
Then, we have
$\tilde{\varphi}: y^{(a)}_m(u_{\mathbf{i}})
\mapsto Y^{(a)}_m(u_{\mathbf{i}})$.
Furthermore, one can prove that
$\tilde{\varphi}: y^{(a)}_m(u)
\mapsto Y^{(a)}_m(u)$
for any $(a,m,u)\in \mathcal{I}'_{\ell+}$
by induction on the forward and backward mutations,
applying the same Y-systems for the both sides.
By the restriction of $\tilde{\varphi}$
to $\mathcal{G}_Y(B,x)$,
we obtain a group homomorphism
$\varphi:
\mathcal{G}_Y(B,x)\rightarrow 
\EuScript{Y}^{\circ}_{\ell}(M_t)_+$,
which is the inverse of $\rho$.
\end{proof}

\section{Cluster algebraic formulation: The case $|I|=2$; $t$ is even}
\label{sect:teven}

In this section we consider the case $|I|=2$ when $t$ is even.
Basically it is parallel to the former case
and we omit proofs.

\subsection{Parity decompositions of T and Y-systems}

For a triplet $(a,m,u)\in \mathcal{I}_{\ell}$,
we reset the `parity conditions' $\mathbf{P}_{+}$ and
$\mathbf{P}_{-}$ by
\begin{align}
\label{eq:BPcond}
\begin{split}
\mathbf{P}_{+}:& \ \mbox{
$tu$ is even
if
$a=1$; $m+tu$ is even if $a=2$},\\
\mathbf{P}_{-}:& \ \mbox{
$tu$ is odd if
$a=1$; $m+tu$ is odd if $a=2$}.
\end{split}
\end{align}
Let $\mathcal{I}_{\ell\varepsilon}$ be
the set of all $(a,m,u):\mathbf{P}_{\varepsilon}$.
Define $\EuScript{T}^{\circ}_{\ell}(M_t)_{\varepsilon}$
($\varepsilon=\pm$)
to be the subring of $\EuScript{T}^{\circ}_{\ell}(M_t)$
generated by
 $T^{(a)}_m(u)$
$((a,m,u)\in \mathcal{I}_{\ell\varepsilon})$.
Then, we have
$\EuScript{T}^{\circ}_{\ell}(M_t)_+
\simeq
\EuScript{T}^{\circ}_{\ell}(M_t)_-
$
by $T^{(a)}_m(u)\mapsto T^{(a)}_m(u+\frac{1}{t})$,
 and the decomposition \eqref{eq:Tfact} holds.

For a triplet $(a,m,u)\in \mathcal{I}_{\ell}$ ,
we set another `parity conditions' $\mathbf{P}'_{+}$ and
$\mathbf{P}'_{-}$ by
\begin{align}
\label{eq:BPcond2}
\begin{split}
\mathbf{P}'_{+}:&\ \mbox{
$tu$ is even if
$a=1$;
 $m+tu$ is odd if $a=2$},\\
\mathbf{P}'_{-}:&\ \mbox{
$tu$ is odd if
$a=1$;
 $m+tu$ is even if $a=2$}.
\end{split}
\end{align}
We have
\begin{align}
(a,m,u):\mathbf{P}'_+ \ \Longleftrightarrow\ 
\textstyle (a,m,u\pm \frac{d_a}{t}):\mathbf{P}_+.
\end{align}
Let $\mathcal{I}'_{\ell\varepsilon}$ be
 the set of all $(a,m,u):\mathbf{P}'_{\varepsilon}$.
Define $\EuScript{Y}^{\circ}_{\ell}(M_t)_{\varepsilon}$
($\varepsilon=\pm$)
to be the subgroup of $\EuScript{Y}^{\circ}_{\ell}(M_t)$
generated by
$Y^{(a)}_m(u)$, $1+Y^{(a)}_m(u)$
$((a,m,u)\in \mathcal{I}'_{\ell\varepsilon})$.
Then, we have
$\EuScript{Y}^{\circ}_{\ell}(M_t)_+
\simeq
\EuScript{Y}^{\circ}_{\ell}(M_t)_-
$
by $Y^{(a)}_m(u)\mapsto Y^{(a)}_m(u+\frac{1}{t})$,
$1+Y^{(a)}_m(u)\mapsto 1+Y^{(a)}_m(u+\frac{1}{t})$,
 and the decomposition \eqref{eq:Yfact} holds.

\subsection{Quiver $Q_{\ell}(M_t)$}
\label{subsect:quivereven}

With the Cartan matrix  $M_t$ and $\ell\geq 2$ we associate
 the quiver $Q_{\ell}(M_t)$.
Again, as a rather general example, the case $t=4$ is given by
 Fig. \ref{fig:quiverB},
where the right columns in  the four quivers $Q_1$,\dots,$Q_4$ are
identified.
Also we assign the empty or filled circle $\circ$/$\bullet$
and the sign +/$-$ to  each vertex as shown.
For a general even $t$,  the quiver $Q_{\ell}(M_t)$ is defined
by naturally extending the case $t=4$.
Even though it looks quite similar to the odd $t$ case
in Fig.\ \ref{fig:quiverG},
there is one important difference due to the parity of $t$;
that is, when $t$ is even, any vertex 
in the right column
of $Q_i$  has the sign `$-$' whenever it is connected
to a vertex in the left column by a  {\em horizontal} arrow.
This is not so when $t$ is odd.

%%%%%%%%%%%%%%%%%%%%%%%%%%%%%
% 
\begin{figure}
\setlength{\unitlength}{0.75pt}
\begin{picture}(360,230)(-100,-20)
%
% first diagram
\put(0,0)
{
\put(0,45){\circle{5}}
\put(0,105){\circle{5}}
\put(0,165){\circle{5}}
\put(30,0){\circle*{5}}
\put(30,15){\circle*{5}}
\put(30,30){\circle*{5}}
\put(30,45){\circle*{5}}
\put(30,60){\circle*{5}}
\put(30,75){\circle*{5}}
\put(30,90){\circle*{5}}
\put(30,105){\circle*{5}}
\put(30,120){\circle*{5}}
\put(30,135){\circle*{5}}
\put(30,150){\circle*{5}}
\put(30,165){\circle*{5}}
\put(30,180){\circle*{5}}
\put(30,195){\circle*{5}}
\put(30,210){\circle*{5}}
%
% vertical arrows
\put(30,3){\vector(0,1){9}}
\put(30,27){\vector(0,-1){9}}
\put(30,33){\vector(0,1){9}}
\put(30,57){\vector(0,-1){9}}
\put(30,63){\vector(0,1){9}}
\put(30,87){\vector(0,-1){9}}
\put(30,93){\vector(0,1){9}}
\put(30,117){\vector(0,-1){9}}
\put(30,123){\vector(0,1){9}}
\put(30,147){\vector(0,-1){9}}
\put(30,153){\vector(0,1){9}}
\put(30,177){\vector(0,-1){9}}
\put(30,183){\vector(0,1){9}}
\put(30,207){\vector(0,-1){9}}
\put(0,102){\vector(0,-1){54}}
\put(0,108){\vector(0,1){54}}
% diagonal arrows
%
\put(3,38){\vector(2,-3){24}}
\put(27,17){\vector(-1,1){23}}
\put(3,43){\vector(2,-1){24}}
\put(27,45){\vector(-1,0){23}}
\put(3,47){\vector(2,1){24}}
\put(27,73){\vector(-1,-1){23}}
\put(3,52){\vector(2,3){24}}
\put(27,105){\vector(-1,0){24}}
\put(3,158){\vector(2,-3){24}}
\put(27,137){\vector(-1,1){23}}
\put(3,163){\vector(2,-1){24}}
\put(27,165){\vector(-1,0){23}}
\put(3,167){\vector(2,1){24}}
\put(27,193){\vector(-1,-1){23}}
\put(3,172){\vector(2,3){24}}
\put(3,-2)
{
\put(-17,45){\small $-$}
\put(-17,105){\small $+$}
\put(-17,165){\small $-$}
}
\put(4,-2)
{
\put(30,0){\small $+$}
\put(30,15){\small $-$}
\put(30,30){\small $+$}
\put(30,45){\small $-$}
\put(30,60){\small $+$}
\put(30,75){\small $-$}
\put(30,90){\small $+$}
\put(30,105){\small $-$}
\put(30,120){\small $+$}
\put(30,135){\small $-$}
\put(30,150){\small $+$}
\put(30,165){\small $-$}
\put(30,180){\small $+$}
\put(30,195){\small $-$}
\put(30,210){\small $+$}
}
}
% second diagram
\put(70,0)
{
\put(0,45){\circle{5}}
\put(0,105){\circle{5}}
\put(0,165){\circle{5}}
\put(30,0){\circle*{5}}
\put(30,15){\circle*{5}}
\put(30,30){\circle*{5}}
\put(30,45){\circle*{5}}
\put(30,60){\circle*{5}}
\put(30,75){\circle*{5}}
\put(30,90){\circle*{5}}
\put(30,105){\circle*{5}}
\put(30,120){\circle*{5}}
\put(30,135){\circle*{5}}
\put(30,150){\circle*{5}}
\put(30,165){\circle*{5}}
\put(30,180){\circle*{5}}
\put(30,195){\circle*{5}}
\put(30,210){\circle*{5}}
%
% vertical arrows
\put(30,3){\vector(0,1){9}}
\put(30,27){\vector(0,-1){9}}
\put(30,33){\vector(0,1){9}}
\put(30,57){\vector(0,-1){9}}
\put(30,63){\vector(0,1){9}}
\put(30,87){\vector(0,-1){9}}
\put(30,93){\vector(0,1){9}}
\put(30,117){\vector(0,-1){9}}
\put(30,123){\vector(0,1){9}}
\put(30,147){\vector(0,-1){9}}
\put(30,153){\vector(0,1){9}}
\put(30,177){\vector(0,-1){9}}
\put(30,183){\vector(0,1){9}}
\put(30,207){\vector(0,-1){9}}
\put(0,48){\vector(0,1){54}}
\put(0,162){\vector(0,-1){54}}
% diagonal arrows
\put(27,17){\vector(-1,1){23}}
\put(3,43){\vector(2,-1){24}}
\put(27,45){\vector(-1,0){23}}
\put(3,47){\vector(2,1){24}}
\put(27,73){\vector(-1,-1){23}}
\put(3,103){\vector(2,-1){24}}
\put(27,105){\vector(-1,0){23}}
\put(3,107){\vector(2,1){24}}
\put(27,137){\vector(-1,1){23}}
\put(3,163){\vector(2,-1){24}}
\put(27,165){\vector(-1,0){23}}
\put(3,167){\vector(2,1){24}}
\put(27,193){\vector(-1,-1){23}}
\put(3,-2)
{
\put(-17,45){\small $+$}
\put(-17,105){\small $-$}
\put(-17,165){\small $+$}
}
\put(4,-2)
{
\put(30,0){\small $+$}
\put(30,15){\small $-$}
\put(30,30){\small $+$}
\put(30,45){\small $-$}
\put(30,60){\small $+$}
\put(30,75){\small $-$}
\put(30,90){\small $+$}
\put(30,105){\small $-$}
\put(30,120){\small $+$}
\put(30,135){\small $-$}
\put(30,150){\small $+$}
\put(30,165){\small $-$}
\put(30,180){\small $+$}
\put(30,195){\small $-$}
\put(30,210){\small $+$}
}
}
% third diagram
\put(140,0)
{
\put(0,45){\circle{5}}
\put(0,105){\circle{5}}
\put(0,165){\circle{5}}
\put(30,0){\circle*{5}}
\put(30,15){\circle*{5}}
\put(30,30){\circle*{5}}
\put(30,45){\circle*{5}}
\put(30,60){\circle*{5}}
\put(30,75){\circle*{5}}
\put(30,90){\circle*{5}}
\put(30,105){\circle*{5}}
\put(30,120){\circle*{5}}
\put(30,135){\circle*{5}}
\put(30,150){\circle*{5}}
\put(30,165){\circle*{5}}
\put(30,180){\circle*{5}}
\put(30,195){\circle*{5}}
\put(30,210){\circle*{5}}
%
% vertical arrows
\put(30,3){\vector(0,1){9}}
\put(30,27){\vector(0,-1){9}}
\put(30,33){\vector(0,1){9}}
\put(30,57){\vector(0,-1){9}}
\put(30,63){\vector(0,1){9}}
\put(30,87){\vector(0,-1){9}}
\put(30,93){\vector(0,1){9}}
\put(30,117){\vector(0,-1){9}}
\put(30,123){\vector(0,1){9}}
\put(30,147){\vector(0,-1){9}}
\put(30,153){\vector(0,1){9}}
\put(30,177){\vector(0,-1){9}}
\put(30,183){\vector(0,1){9}}
\put(30,207){\vector(0,-1){9}}
\put(0,102){\vector(0,-1){54}}
\put(0,108){\vector(0,1){54}}
% diagonal arrows
\put(3,43){\vector(2,-1){24}}
\put(27,45){\vector(-1,0){23}}
\put(3,47){\vector(2,1){24}}
\put(27,77){\vector(-1,1){23}}
\put(3,103){\vector(2,-1){24}}
\put(27,105){\vector(-1,0){23}}
\put(3,107){\vector(2,1){24}}
\put(27,133){\vector(-1,-1){23}}
\put(3,163){\vector(2,-1){24}}
\put(27,165){\vector(-1,0){23}}
\put(3,167){\vector(2,1){24}}
\put(3,-2)
{
\put(-17,45){\small $-$}
\put(-17,105){\small $+$}
\put(-17,165){\small $-$}
}
\put(4,-2)
{
\put(30,0){\small $+$}
\put(30,15){\small $-$}
\put(30,30){\small $+$}
\put(30,45){\small $-$}
\put(30,60){\small $+$}
\put(30,75){\small $-$}
\put(30,90){\small $+$}
\put(30,105){\small $-$}
\put(30,120){\small $+$}
\put(30,135){\small $-$}
\put(30,150){\small $+$}
\put(30,165){\small $-$}
\put(30,180){\small $+$}
\put(30,195){\small $-$}
\put(30,210){\small $+$}
}
}
%
% fourth diagram
\put(210,0)
{
\put(0,45){\circle{5}}
\put(0,105){\circle{5}}
\put(0,165){\circle{5}}
\put(30,0){\circle*{5}}
\put(30,15){\circle*{5}}
\put(30,30){\circle*{5}}
\put(30,45){\circle*{5}}
\put(30,60){\circle*{5}}
\put(30,75){\circle*{5}}
\put(30,90){\circle*{5}}
\put(30,105){\circle*{5}}
\put(30,120){\circle*{5}}
\put(30,135){\circle*{5}}
\put(30,150){\circle*{5}}
\put(30,165){\circle*{5}}
\put(30,180){\circle*{5}}
\put(30,195){\circle*{5}}
\put(30,210){\circle*{5}}
%
% vertical arrows
\put(30,3){\vector(0,1){9}}
\put(30,27){\vector(0,-1){9}}
\put(30,33){\vector(0,1){9}}
\put(30,57){\vector(0,-1){9}}
\put(30,63){\vector(0,1){9}}
\put(30,87){\vector(0,-1){9}}
\put(30,93){\vector(0,1){9}}
\put(30,117){\vector(0,-1){9}}
\put(30,123){\vector(0,1){9}}
\put(30,147){\vector(0,-1){9}}
\put(30,153){\vector(0,1){9}}
\put(30,177){\vector(0,-1){9}}
\put(30,183){\vector(0,1){9}}
\put(30,207){\vector(0,-1){9}}
\put(0,48){\vector(0,1){54}}
\put(0,162){\vector(0,-1){54}}
% diagonal arrows
\put(27,45){\vector(-1,0){23}}
\put(3,98){\vector(2,-3){24}}
\put(27,77){\vector(-1,1){23}}
\put(3,103){\vector(2,-1){24}}
\put(27,105){\vector(-1,0){23}}
\put(3,107){\vector(2,1){24}}
\put(27,133){\vector(-1,-1){23}}
\put(3,112){\vector(2,3){24}}
\put(27,165){\vector(-1,0){23}}
\put(3,-2)
{
\put(-17,45){\small $+$}
\put(-17,105){\small $-$}
\put(-17,165){\small $+$}
}
\put(4,-2)
{
\put(30,0){\small $+$}
\put(30,15){\small $-$}
\put(30,30){\small $+$}
\put(30,45){\small $-$}
\put(30,60){\small $+$}
\put(30,75){\small $-$}
\put(30,90){\small $+$}
\put(30,105){\small $-$}
\put(30,120){\small $+$}
\put(30,135){\small $-$}
\put(30,150){\small $+$}
\put(30,165){\small $-$}
\put(30,180){\small $+$}
\put(30,195){\small $-$}
\put(30,210){\small $+$}
}
}
\put(10,-20){$Q_1$}
\put(80,-20){$Q_2$}
\put(150,-20){$Q_3$}
\put(220,-20){$Q_4$}
\put(-50,103){${\scriptstyle\ell -1}\left\{ \makebox(0,60){}\right.$}
\put(260,103){$\left. \makebox(0,112){}\right\}{\scriptstyle t\ell -1}$}
\end{picture}
%
%%%%%%%%%%%%%%%%%%%%%%%%%%%%%%%%%%%
%
\begin{picture}(360,305)(-100,-20)
%
% first diagram
\put(0,0)
{
\put(0,45){\circle{5}}
\put(0,105){\circle{5}}
\put(0,165){\circle{5}}
\put(0,225){\circle{5}}
\put(30,0){\circle*{5}}
\put(30,15){\circle*{5}}
\put(30,30){\circle*{5}}
\put(30,45){\circle*{5}}
\put(30,60){\circle*{5}}
\put(30,75){\circle*{5}}
\put(30,90){\circle*{5}}
\put(30,105){\circle*{5}}
\put(30,120){\circle*{5}}
\put(30,135){\circle*{5}}
\put(30,150){\circle*{5}}
\put(30,165){\circle*{5}}
\put(30,180){\circle*{5}}
\put(30,195){\circle*{5}}
\put(30,210){\circle*{5}}
\put(30,225){\circle*{5}}
\put(30,240){\circle*{5}}
\put(30,255){\circle*{5}}
\put(30,270){\circle*{5}}
%
% vertical arrows
\put(30,3){\vector(0,1){9}}
\put(30,27){\vector(0,-1){9}}
\put(30,33){\vector(0,1){9}}
\put(30,57){\vector(0,-1){9}}
\put(30,63){\vector(0,1){9}}
\put(30,87){\vector(0,-1){9}}
\put(30,93){\vector(0,1){9}}
\put(30,117){\vector(0,-1){9}}
\put(30,123){\vector(0,1){9}}
\put(30,147){\vector(0,-1){9}}
\put(30,153){\vector(0,1){9}}
\put(30,177){\vector(0,-1){9}}
\put(30,183){\vector(0,1){9}}
\put(30,207){\vector(0,-1){9}}
\put(30,213){\vector(0,1){9}}
\put(30,237){\vector(0,-1){9}}
\put(30,243){\vector(0,1){9}}
\put(30,267){\vector(0,-1){9}}
\put(0,102){\vector(0,-1){54}}
\put(0,108){\vector(0,1){54}}
\put(0,222){\vector(0,-1){54}}
% diagonal arrows
\put(3,38){\vector(2,-3){24}}
\put(27,17){\vector(-1,1){23}}
\put(3,43){\vector(2,-1){24}}
\put(27,45){\vector(-1,0){23}}
\put(3,47){\vector(2,1){24}}
\put(27,73){\vector(-1,-1){23}}
\put(3,52){\vector(2,3){24}}
\put(27,105){\vector(-1,0){24}}
\put(3,158){\vector(2,-3){24}}
\put(27,137){\vector(-1,1){23}}
\put(3,163){\vector(2,-1){24}}
\put(27,165){\vector(-1,0){23}}
\put(3,167){\vector(2,1){24}}
\put(27,193){\vector(-1,-1){23}}
\put(3,172){\vector(2,3){24}}
\put(27,225){\vector(-1,0){24}}
\put(3,-2)
{
\put(-17,45){\small $-$}
\put(-17,105){\small $+$}
\put(-17,165){\small $-$}
\put(-17,225){\small $+$}
}
\put(4,-2)
{
\put(30,0){\small $+$}
\put(30,15){\small $-$}
\put(30,30){\small $+$}
\put(30,45){\small $-$}
\put(30,60){\small $+$}
\put(30,75){\small $-$}
\put(30,90){\small $+$}
\put(30,105){\small $-$}
\put(30,120){\small $+$}
\put(30,135){\small $-$}
\put(30,150){\small $+$}
\put(30,165){\small $-$}
\put(30,180){\small $+$}
\put(30,195){\small $-$}
\put(30,210){\small $+$}
\put(30,225){\small $-$}
\put(30,240){\small $+$}
\put(30,255){\small $-$}
\put(30,270){\small $+$}
}
}
% second diagram
\put(70,0)
{
\put(0,45){\circle{5}}
\put(0,105){\circle{5}}
\put(0,165){\circle{5}}
\put(0,225){\circle{5}}
\put(30,0){\circle*{5}}
\put(30,15){\circle*{5}}
\put(30,30){\circle*{5}}
\put(30,45){\circle*{5}}
\put(30,60){\circle*{5}}
\put(30,75){\circle*{5}}
\put(30,90){\circle*{5}}
\put(30,105){\circle*{5}}
\put(30,120){\circle*{5}}
\put(30,135){\circle*{5}}
\put(30,150){\circle*{5}}
\put(30,165){\circle*{5}}
\put(30,180){\circle*{5}}
\put(30,195){\circle*{5}}
\put(30,210){\circle*{5}}
\put(30,225){\circle*{5}}
\put(30,240){\circle*{5}}
\put(30,255){\circle*{5}}
\put(30,270){\circle*{5}}
%
% vertical arrows
\put(30,3){\vector(0,1){9}}
\put(30,27){\vector(0,-1){9}}
\put(30,33){\vector(0,1){9}}
\put(30,57){\vector(0,-1){9}}
\put(30,63){\vector(0,1){9}}
\put(30,87){\vector(0,-1){9}}
\put(30,93){\vector(0,1){9}}
\put(30,117){\vector(0,-1){9}}
\put(30,123){\vector(0,1){9}}
\put(30,147){\vector(0,-1){9}}
\put(30,153){\vector(0,1){9}}
\put(30,177){\vector(0,-1){9}}
\put(30,183){\vector(0,1){9}}
\put(30,207){\vector(0,-1){9}}
\put(30,213){\vector(0,1){9}}
\put(30,237){\vector(0,-1){9}}
\put(30,243){\vector(0,1){9}}
\put(30,267){\vector(0,-1){9}}
\put(0,48){\vector(0,1){54}}
\put(0,162){\vector(0,-1){54}}
\put(0,168){\vector(0,1){54}}
% diagonal arrows
\put(27,17){\vector(-1,1){23}}
\put(3,43){\vector(2,-1){24}}
\put(27,45){\vector(-1,0){23}}
\put(3,47){\vector(2,1){24}}
\put(27,73){\vector(-1,-1){23}}
\put(3,103){\vector(2,-1){24}}
\put(27,105){\vector(-1,0){23}}
\put(3,107){\vector(2,1){24}}
\put(27,137){\vector(-1,1){23}}
\put(3,163){\vector(2,-1){24}}
\put(27,165){\vector(-1,0){23}}
\put(3,167){\vector(2,1){24}}
\put(27,193){\vector(-1,-1){23}}
\put(3,223){\vector(2,-1){24}}
\put(27,225){\vector(-1,0){23}}
\put(3,227){\vector(2,1){24}}
\put(3,-2)
{
\put(-17,45){\small $+$}
\put(-17,105){\small $-$}
\put(-17,165){\small $+$}
\put(-17,225){\small $-$}
}
\put(4,-2)
{
\put(30,0){\small $+$}
\put(30,15){\small $-$}
\put(30,30){\small $+$}
\put(30,45){\small $-$}
\put(30,60){\small $+$}
\put(30,75){\small $-$}
\put(30,90){\small $+$}
\put(30,105){\small $-$}
\put(30,120){\small $+$}
\put(30,135){\small $-$}
\put(30,150){\small $+$}
\put(30,165){\small $-$}
\put(30,180){\small $+$}
\put(30,195){\small $-$}
\put(30,210){\small $+$}
\put(30,225){\small $-$}
\put(30,240){\small $+$}
\put(30,255){\small $-$}
\put(30,270){\small $+$}
}
}
% third diagram
\put(140,0)
{
\put(0,45){\circle{5}}
\put(0,105){\circle{5}}
\put(0,165){\circle{5}}
\put(0,225){\circle{5}}
\put(30,0){\circle*{5}}
\put(30,15){\circle*{5}}
\put(30,30){\circle*{5}}
\put(30,45){\circle*{5}}
\put(30,60){\circle*{5}}
\put(30,75){\circle*{5}}
\put(30,90){\circle*{5}}
\put(30,105){\circle*{5}}
\put(30,120){\circle*{5}}
\put(30,135){\circle*{5}}
\put(30,150){\circle*{5}}
\put(30,165){\circle*{5}}
\put(30,180){\circle*{5}}
\put(30,195){\circle*{5}}
\put(30,210){\circle*{5}}
\put(30,225){\circle*{5}}
\put(30,240){\circle*{5}}
\put(30,255){\circle*{5}}
\put(30,270){\circle*{5}}
%
% vertical arrows
\put(30,3){\vector(0,1){9}}
\put(30,27){\vector(0,-1){9}}
\put(30,33){\vector(0,1){9}}
\put(30,57){\vector(0,-1){9}}
\put(30,63){\vector(0,1){9}}
\put(30,87){\vector(0,-1){9}}
\put(30,93){\vector(0,1){9}}
\put(30,117){\vector(0,-1){9}}
\put(30,123){\vector(0,1){9}}
\put(30,147){\vector(0,-1){9}}
\put(30,153){\vector(0,1){9}}
\put(30,177){\vector(0,-1){9}}
\put(30,183){\vector(0,1){9}}
\put(30,207){\vector(0,-1){9}}
\put(30,213){\vector(0,1){9}}
\put(30,237){\vector(0,-1){9}}
\put(30,243){\vector(0,1){9}}
\put(30,267){\vector(0,-1){9}}
\put(0,102){\vector(0,-1){54}}
\put(0,108){\vector(0,1){54}}
\put(0,222){\vector(0,-1){54}}
% diagonal arrows
\put(3,43){\vector(2,-1){24}}
\put(27,45){\vector(-1,0){23}}
\put(3,47){\vector(2,1){24}}
\put(27,77){\vector(-1,1){23}}
\put(3,103){\vector(2,-1){24}}
\put(27,105){\vector(-1,0){23}}
\put(3,107){\vector(2,1){24}}
\put(27,133){\vector(-1,-1){23}}
\put(3,163){\vector(2,-1){24}}
\put(27,165){\vector(-1,0){23}}
\put(3,167){\vector(2,1){24}}
\put(27,197){\vector(-1,1){23}}
\put(3,223){\vector(2,-1){24}}
\put(27,225){\vector(-1,0){23}}
\put(3,227){\vector(2,1){24}}
\put(27,253){\vector(-1,-1){23}}
\put(3,-2)
{
\put(-17,45){\small $-$}
\put(-17,105){\small $+$}
\put(-17,165){\small $-$}
\put(-17,225){\small $+$}
}
\put(4,-2)
{
\put(30,0){\small $+$}
\put(30,15){\small $-$}
\put(30,30){\small $+$}
\put(30,45){\small $-$}
\put(30,60){\small $+$}
\put(30,75){\small $-$}
\put(30,90){\small $+$}
\put(30,105){\small $-$}
\put(30,120){\small $+$}
\put(30,135){\small $-$}
\put(30,150){\small $+$}
\put(30,165){\small $-$}
\put(30,180){\small $+$}
\put(30,195){\small $-$}
\put(30,210){\small $+$}
\put(30,225){\small $-$}
\put(30,240){\small $+$}
\put(30,255){\small $-$}
\put(30,270){\small $+$}
}
}
%
% fourth diagram
\put(210,0)
{
\put(0,45){\circle{5}}
\put(0,105){\circle{5}}
\put(0,165){\circle{5}}
\put(0,225){\circle{5}}
\put(30,0){\circle*{5}}
\put(30,15){\circle*{5}}
\put(30,30){\circle*{5}}
\put(30,45){\circle*{5}}
\put(30,60){\circle*{5}}
\put(30,75){\circle*{5}}
\put(30,90){\circle*{5}}
\put(30,105){\circle*{5}}
\put(30,120){\circle*{5}}
\put(30,135){\circle*{5}}
\put(30,150){\circle*{5}}
\put(30,165){\circle*{5}}
\put(30,180){\circle*{5}}
\put(30,195){\circle*{5}}
\put(30,210){\circle*{5}}
\put(30,225){\circle*{5}}
\put(30,240){\circle*{5}}
\put(30,255){\circle*{5}}
\put(30,270){\circle*{5}}
%
% vertical arrows
\put(30,3){\vector(0,1){9}}
\put(30,27){\vector(0,-1){9}}
\put(30,33){\vector(0,1){9}}
\put(30,57){\vector(0,-1){9}}
\put(30,63){\vector(0,1){9}}
\put(30,87){\vector(0,-1){9}}
\put(30,93){\vector(0,1){9}}
\put(30,117){\vector(0,-1){9}}
\put(30,123){\vector(0,1){9}}
\put(30,147){\vector(0,-1){9}}
\put(30,153){\vector(0,1){9}}
\put(30,177){\vector(0,-1){9}}
\put(30,183){\vector(0,1){9}}
\put(30,207){\vector(0,-1){9}}
\put(30,213){\vector(0,1){9}}
\put(30,237){\vector(0,-1){9}}
\put(30,243){\vector(0,1){9}}
\put(30,267){\vector(0,-1){9}}
\put(0,102){\vector(0,-1){54}}
\put(0,108){\vector(0,1){54}}
\put(0,222){\vector(0,-1){54}}
% diagonal arrows
\put(27,45){\vector(-1,0){23}}
\put(3,98){\vector(2,-3){24}}
\put(27,77){\vector(-1,1){23}}
\put(3,103){\vector(2,-1){24}}
\put(27,105){\vector(-1,0){23}}
\put(3,107){\vector(2,1){24}}
\put(27,133){\vector(-1,-1){23}}
\put(3,112){\vector(2,3){24}}
\put(27,165){\vector(-1,0){23}}
\put(3,218){\vector(2,-3){24}}
\put(27,197){\vector(-1,1){23}}
\put(3,223){\vector(2,-1){24}}
\put(27,225){\vector(-1,0){23}}
\put(3,227){\vector(2,1){24}}
\put(27,253){\vector(-1,-1){23}}
\put(3,232){\vector(2,3){24}}
\put(3,-2)
{
\put(-17,45){\small $+$}
\put(-17,105){\small $-$}
\put(-17,165){\small $+$}
\put(-17,225){\small $-$}
}
\put(4,-2)
{
\put(30,0){\small $+$}
\put(30,15){\small $-$}
\put(30,30){\small $+$}
\put(30,45){\small $-$}
\put(30,60){\small $+$}
\put(30,75){\small $-$}
\put(30,90){\small $+$}
\put(30,105){\small $-$}
\put(30,120){\small $+$}
\put(30,135){\small $-$}
\put(30,150){\small $+$}
\put(30,165){\small $-$}
\put(30,180){\small $+$}
\put(30,195){\small $-$}
\put(30,210){\small $+$}
\put(30,225){\small $-$}
\put(30,240){\small $+$}
\put(30,255){\small $-$}
\put(30,270){\small $+$}
}
}
\put(10,-20){$Q_1$}
\put(80,-20){$Q_2$}
\put(150,-20){$Q_3$}
\put(220,-20){$Q_4$}
\put(-50,135){${\scriptstyle\ell -1}\left\{ \makebox(0,93){}\right.$}
\put(260,135){$\left. \makebox(0,140){}\right\}{\scriptstyle t\ell -1}$}
\end{picture}
\caption{The quiver $Q_{\ell}(M_t)$ with $t=4$  for even $\ell$
(upper) and for odd $\ell$ (lower),
where we identify the right columns in all the  quivers
$Q_1$, \dots, $Q_4$.}
\label{fig:quiverB}
\end{figure}
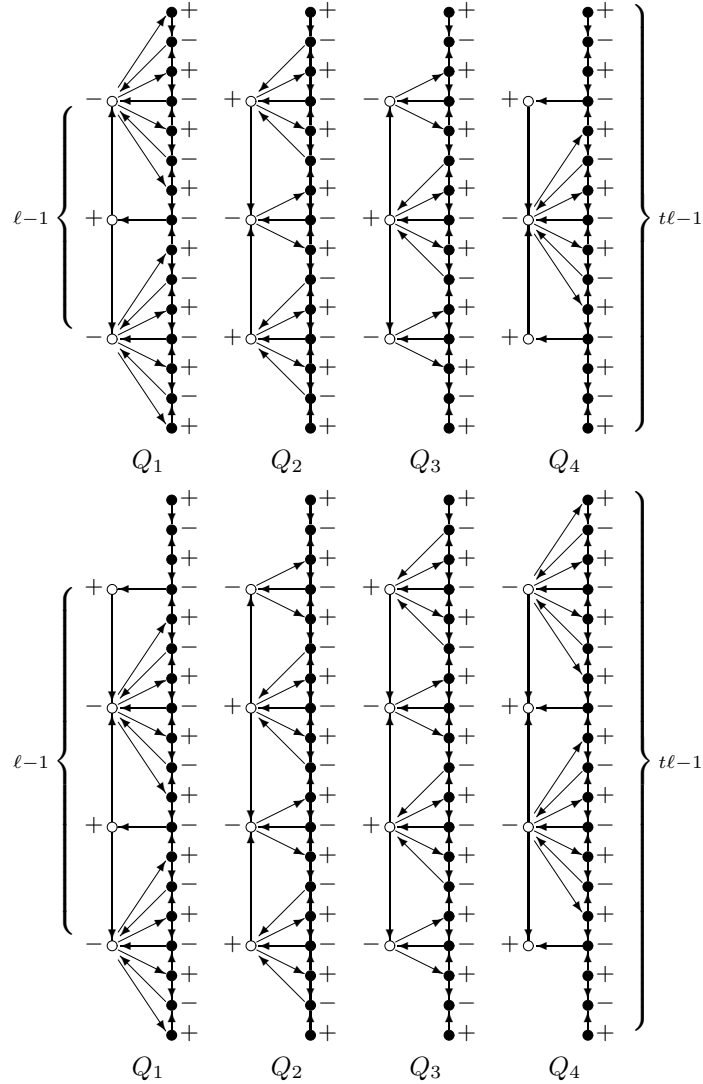

Let us choose  the index set $\mathbf{I}$
of the vertices of $Q_{\ell}(M_t)$
as before so that $\mathbf{i}=(i,i')\in \mathbf{I}$ represents
the vertex 
at the $i'$th row (from the bottom)
of the left column in $Q_i$ for $i=1,\dots,t$,
and the one of the right column in any quiver
for $i=t+1$.
We use the same notations
$\mathbf{I}^{\circ}_{\pm,k}$
$\mathbf{I}^{\bullet}_{\pm}$ as before.
For $k,k'\in \{1,\dots,t\}$, $k\neq k'$,
 let $\mathbf{I}^{\circ}_{\pm,k,k'}=
\mathbf{I}^{\circ}_{\pm,k}\sqcup 
\mathbf{I}^{\circ}_{\pm,k'}$.
We define composite mutations,
\begin{align}
\label{eq:Bmupm2}
\mu^{\circ}_{+,k,k'}=\prod_{\mathbf{i}\in\mathbf{I}^{\circ}_{+,k,k'}}
\mu_{\mathbf{i}},
\quad
\mu^{\circ}_{-,k,k'}=\prod_{\mathbf{i}\in\mathbf{I}^{\circ}_{-,k,k'}}
\mu_{\mathbf{i}},
\quad
\mu^{\bullet}_+=\prod_{\mathbf{i}\in\mathbf{I}^{\bullet}_+}
\mu_{\mathbf{i}},
\quad
\mu^{\bullet}_-=\prod_{\mathbf{i}\in\mathbf{I}^{\bullet}_-}
\mu_{\mathbf{i}}.
\end{align}

\begin{Lemma}
\label{lem:BQmut}
Let $Q(0):=Q_{\ell}(M_t)$.
We have the following periodic sequence of mutations of quivers:
\begin{align}
\label{eq:BB2}
\begin{matrix}
% 1st
Q(0)
& 
\displaystyle
\mathop{\longleftrightarrow}^{\mu^{\bullet}_+
\mu^{\circ}_{+,1,t}}
& Q(\frac{1}{t})
&
\displaystyle
\mathop{\longleftrightarrow}^{\mu^{\bullet}_-
}
&
Q(\frac{2}{t})
&
\displaystyle
\mathop{\longleftrightarrow}^{\mu^{\bullet}_+
\mu^{\circ}_{+,3,t-2}}&
Q(\frac{3}{t})
&
\displaystyle
\mathop{\longleftrightarrow}^{\mu^{\bullet}_-
}
&
Q(\frac{4}{t})
\\
% 2nd
&
\displaystyle
\mathop{\longleftrightarrow}^{\mu^{\bullet}_+
\mu^{\circ}_{+,5,t-4}}
&
&
%\boldsymbol{\nu}_{t-2}(Q)
\cdots
&&
\displaystyle
\mathop{\longleftrightarrow}^{\mu^{\bullet}_+
\mu^{\circ}_{+,t-1,2}}
&
Q(\frac{t-1}{t})
&
\displaystyle
\mathop{\longleftrightarrow}^{\mu^{\bullet}_-
}
&
Q(1)
&
\\
% 3rd
& 
\displaystyle
\mathop{\longleftrightarrow}^{\mu^{\bullet}_+
\mu^{\circ}_{-,1,t}}
&
Q(\frac{t+1}{t})
&
\displaystyle
\mathop{\longleftrightarrow}^{\mu^{\bullet}_-
}
&
Q(\frac{t+2}{t})
&
\displaystyle
\mathop{\longleftrightarrow}^{\mu^{\bullet}_+
\mu^{\circ}_{-,3,t-2}}
&
Q(\frac{t+3}{t})
&
\displaystyle
\mathop{\longleftrightarrow}^{\mu^{\bullet}_-
}
&
Q(\frac{t+4}{t})
\\
% 4th
&
\displaystyle
\mathop{\longleftrightarrow}^{\mu^{\bullet}_+
\mu^{\circ}_{-,5,t-4}}
&
&
\cdots
&
&
\displaystyle
\mathop{\longleftrightarrow}^{\mu^{\bullet}_+
\mu^{\circ}_{-,t-1,2}}
&
Q(\frac{2t-1}{t})
&
\displaystyle
\mathop{\longleftrightarrow}^{\mu^{\bullet}_-
}
&
Q(2)=Q(0).
&
\\
\end{matrix}
\end{align}
Here, the quiver $Q(p/t)$ $(p=1,\dots,2t)$ is defined by
\begin{align}
Q(p/t):=
\begin{cases}
\tilde{w}_p(Q)^{\mathrm{op}}&\mbox{$p$: odd}\\
\tilde{w}_p(Q)&\mbox{$p$: even},
\end{cases}
\end{align}
and $w_p$ is a permutation
of $\{1,\dots,t\}$ defined by
\begin{align}
w_p&=
\begin{cases}
 r_+ r_- \cdots r_+ \ \mbox{\em{($p$ terms)}} &\mbox{$p$: odd}\\
 r_+ r_- \cdots r_- \ \mbox{\em{($p$ terms)}}& \mbox{$p$: even},
\end{cases}\\
r_+&=(23)(45)\cdots (r-2,r-1),\quad
r_-=(12)(34)\cdots (r-1,r),
\end{align}
where $(ij)$ is the transposition of $i$ and $j$.
\end{Lemma}

\begin{Example}
The mutation sequence
\eqref{eq:BB2} for $t=4$ is explicitly given 
in Fig. \ref{fig:labelxB1}.
where only a part of each quiver is presented as before.
\end{Example}

\begin{figure}[t]
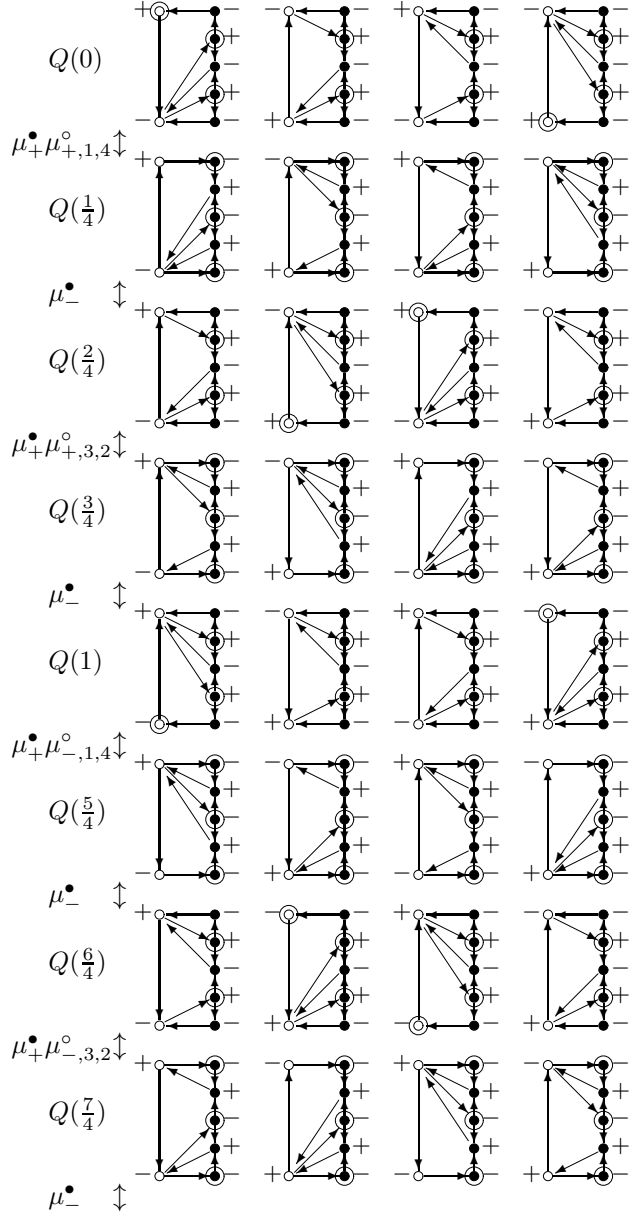

\setlength{\unitlength}{0.7pt}
%%%%%%%%%%%%%%%%%%%%%%%%%%%%%%%%%%%%%%%

%
%
%%%%%%%%%%%%%%%%%%%%%%%%%%%%%%%%%%%
\caption{
The mutation sequence of the quiver $Q_{\ell}(M_t)$
in  \eqref{eq:BB2} for $t=4$.
}
\label{fig:labelxB1}
\end{figure}

\subsection{Embedding maps}

Let $B=B_{\ell}(M_t)$
be the corresponding skew-symmetric matrix 
to the quiver $Q_{\ell}(M_t)$ for even $t$.
Let $\mathcal{A}(B,x,y)$ 
be the cluster algebra
 with coefficients
in the universal
semifield, and let $\mathcal{G}(B,y)$
be the coefficient group associated with $\mathcal{A}(B,x,y)$
as before.

In view of Lemma \ref{lem:BQmut}
we set $x(0)=x$, $y(0)=y$ and define 
clusters $x(u)=(x_{\mathbf{i}}(u))_{\mathbf{i}\in \mathbf{I}}$
 ($u\in \frac{1}{t}\mathbb{Z}$)
 and coefficient tuples $y(u)=(y_\mathbf{i}(u))_{\mathbf{i}\in \mathbf{I}}$
 ($u\in \frac{1}{t}\mathbb{Z}$)
by the sequence of mutations
\begin{align}
\label{eq:Bmutseq}
\begin{matrix}
% 1st
\cdots
&
\displaystyle
\mathop{\longleftrightarrow}^{\mu^{\bullet}_-
}
&
(B(0),x(0),y(0))
& 
\displaystyle
\mathop{\longleftrightarrow}^{\mu^{\bullet}_+
\mu^{\circ}_{+,1,t}}
& (B(\frac{1}{t}),x(\frac{1}{t}),y(\frac{1}{t}))\\
&
\displaystyle
\mathop{\longleftrightarrow}^{\mu^{\bullet}_-
}
&
\cdots
&
\displaystyle
\mathop{\longleftrightarrow}^{\mu^{\bullet}_-
}
&
(B(2),x(2),y(2))
&
\displaystyle
\mathop{\longleftrightarrow}^{\mu^{\bullet}_+
\mu^{\circ}_{+,1,t}}
&
\cdots,
\\
\end{matrix}
\end{align}
where $B(u)$ is the skew-symmetric matrix corresponding to
$Q(u)$.

For  $(\mathbf{i},u)\in
 \mathbf{I}\times \frac{1}{t}\mathbb{Z}$,
we set the parity condition $\mathbf{p}_+$ by
\begin{align}
\mathbf{p}_+:&
\begin{cases}
 \mathbf{i}\in 
\mathbf{I}^{\bullet}_+ \sqcup 
\mathbf{I}^{\circ}_{+,p+1,t-p}
& u\equiv \frac{p}{t}, 0\leq p\leq t-1, \mbox{$p$: even}\\
 \mathbf{i}\in 
\mathbf{I}^{\bullet}_+ \sqcup 
\mathbf{I}^{\circ}_{-,p+1-t,2t-p}
& u\equiv \frac{p}{t}, t\leq p\leq 2t-1, \mbox{$p$: even}\\
 \mathbf{i}\in 
\mathbf{I}^{\bullet}_- 
& u\equiv \frac{p}{t}, 0\leq p\leq 2t-1, \mbox{$p$: odd},\\
\end{cases}
\end{align}
where $\equiv$ is modulo $2\mathbb{Z}$.
Again, each $(\mathbf{i},u):\mathbf{p}_+$
is a mutation point of
\eqref{eq:BB2} in the forward
direction of $u$.

\begin{Lemma}
Below $\equiv$ means the equivalence modulo $2\mathbb{Z}$.
\par
(i)
The map
$g: \mathcal{I}_{\ell+}\rightarrow 
 \{ (\mathbf{i},u): \mathbf{p}_+
\}$
\begin{align}
\textstyle
(a,m,u-\frac{d_a}{t})\mapsto 
\begin{cases}
((2j+1,m),u)& a= 1;
m+u\equiv \frac{2j}{t}\\
&(j=0,1,\dots,t/2-1)\\
((2t-2j,m),u)& a= 1;
m+u\equiv \frac{2j}{t}\\
&(j=t/2,\dots,t-1)\\
((t+1,m),u)& \mbox{\rm  $a=2$}\\
\end{cases}
\end{align}
is a bijection.
\par
(ii)
The map
$g': \mathcal{I}'_{\ell+}\rightarrow 
 \{ (\mathbf{i},u): \mathbf{p}_+
\}$
\begin{align}
(a,m,u)\mapsto 
\begin{cases}
((2j+1,m),u)& a= 1;
m+u\equiv \frac{2j}{t}\\
&(j=0,1,\dots,t/2-1)\\
((2t-2j,m),u)& a= 1;
m+u\equiv \frac{2j}{t}\\
&(j=t/2,\dots,t-1)\\
((t+1,m),u)& \mbox{\rm  $a=2$}\\
\end{cases}
\end{align}
is a bijection.
\end{Lemma}

\subsection{T-system, Y-system, and cluster algebra}

All the properties depending on the parity of $t$ are
now
absorbed in the quiver $Q_{\ell}(M_t)$,
the mutation sequence \eqref{eq:Bmutseq},
and the embedding maps $g$ and $g'$ in Lemma \ref{lem:gmap}.
Lemmas \ref{lem:Gx2},
\ref{lem:Gy2}, and Theorems \ref{thm:GTiso},
\ref{thm:GYiso} are true for even $t$.

\section{Cluster algebraic formulation: Tree case}
\label{sect:tree}

In this section we extend
Theorems \ref{thm:GTiso}
and \ref{thm:GYiso} to any
tamely laced  Cartan matrix $C$
whose Dynkin diagram is a tree,
by patching parity conditions and
 quivers introduced in Secs.~\ref{sect:todd} and \ref{sect:teven}.
This is an intermediate step for treating the most general
case in Sec.~\ref{sect:general}.

\subsection{Parity decompositions 
of T and Y-systems}

\label{subsect:parity}
Throughout this section
we assume that $C$ is a 
tamely laced and indecomposable Cartan matrix
whose Dynkin diagram $X(C)$ is a tree,
i.e., without cycles.

We decompose the index set $I$ of $X(C)$ into
two parts $I=I_+\sqcup I_-$ such that
the following two rules are satisfied:

\begin{itemize}
\item[(I)] If $a$ and $b$ are adjacent in $X(C)$
and both $d_a$ and $d_b$ are odd,
then either $a\in I_+$, $b\in I_-$ or
$a\in I_+$, $b\in I_-$ holds.

\item[(II)] If $a$ and $b$ are adjacent in $X(C)$
and at least one of $d_a$ and $d_b$ is even,
then either $a,b\in I_+$ or
$a,b\in I_-$ holds.
\end{itemize}

%There are  two choices of such decompositions.
%We fix one of them.

To each $a\in I$ with $d_a$ even,
we also attach the `color' $c_{a}= \alpha$ or $\beta$
satisfying the following condition:

\begin{itemize}
\item[(III)] If $a$ and $b$ are adjacent in $X(C)$,
then $c_a\neq c_b$.
\end{itemize}

%There are several ways to do so in general.
%We fix one of them.
See Fig.~\ref{fig:parity} for an example.
(The coloring is not used in this subsection.)

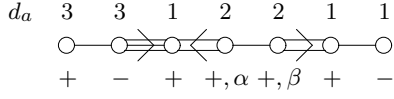
\begin{figure}[t]
\begin{picture}(120,31)(-105,-15)
%
% A_r
%
% B_r
\put(0,0){\circle{6}}
\put(20,0){\circle{6}}
\put(40,0){\circle{6}}
\put(60,0){\circle{6}}
\put(80,0){\circle{6}}
\put(100,0){\circle{6}}
\put(120,0){\circle{6}}
\drawline(3,0)(17,0)
\drawline(22,-2)(38,-2)
\drawline(22,2)(38,2)
\drawline(23,0)(37,0)
\drawline(27,6)(33,0)
\drawline(27,-6)(33,0)
\drawline(42,-2)(58,-2)
\drawline(42,2)(58,2)
\drawline(47,0)(53,6)
\drawline(47,0)(53,-6)
\drawline(63,0)(77,0)
\drawline(82,-2)(98,-2)
\drawline(82,2)(98,2)
\drawline(87,6)(93,0)
\drawline(87,-6)(93,0)
\drawline(103,0)(117,0)
\put(-22,10){\small $d_a$}
\put(-2,10){\small $3$}
\put(18,10){\small $3$}
\put(38,10){\small $1$}
\put(58,10){\small $2$}
\put(78,10){\small $2$}
\put(98,10){\small $1$}
\put(118,10){\small $1$}
\put(-3,-15){\small $+$}
\put(17,-15){\small $-$}
\put(37,-15){\small $+$}
\put(52,-15){\small $+,\alpha$}
\put(72,-15){\small $+,\beta$}
\put(97,-15){\small $+$}
\put(117,-15){\small $-$}
\end{picture}
\caption{Example of a decomposition 
and a coloring of $I$.}
\label{fig:parity}
\end{figure}

For a triplet $(a,m,u)\in \mathcal{I}_{\ell}$,
we set the parity conditions $\mathbf{Q}_{+}$ 
as follows.
\begin{align}
\mathbf{Q}_+&=
\begin{cases}
\mbox{\rm (o+)  $m+tu$ is even}& \mbox{$d_a$ is odd, $a\in I_+$}\\
\mbox{\rm (o$-$) $m+tu$ is odd}& \mbox{$d_a$ is odd, $a\in I_-$}\\
\mbox{\rm (e+)  $tu$ is odd}& \mbox{$d_a$ is even, $a\in I_+$}\\
\mbox{\rm (e$-$) $tu$ is even}& \mbox{$d_a$ is even, $a\in I_-$}\\
\end{cases}
\end{align}
Let $\mathbf{Q}_{-}$ be the negation of $\mathbf{Q}_{+}$.
Suppose that $a$ and $b$ in $I$ with $d_a\geq d_b$
are adjacent in $X(C)$.
Due to the condition \eqref{eq:Ccond1}, we have 
four possibilities:
(i) $d_a$ is odd and $d_b=1$,
(ii) $d_a=d_b$, and $d_a$ is odd and not 1,
(iii) $d_a$ is even and $d_b=1$,
(iv) $d_a=d_b$, and $d_a$ is even.
For (i), the condition $\mathbf{Q}_+$ is compatible with
$\mathbf{P}_{\pm}$ in \eqref{eq:GPcond} with $t=d_a$ therein.
For (iii),
 the condition $\mathbf{Q}_+$ is compatible with
$\mathbf{P}_{\pm}$ in \eqref{eq:BPcond} with $t=d_a$ therein.
For (ii) and (iv),
one can directly check that
 the condition $\mathbf{Q}_+$ is compatible with
\eqref{eq:T1}.
Therefore, we have the parity decomposition
\begin{align}
%\label{eq:Tfact}
\EuScript{T}^{\circ}_{\ell}(C)
\simeq
\EuScript{T}^{\circ}_{\ell}(C)_+
\otimes_{\mathbb{Z}}
\EuScript{T}^{\circ}_{\ell}(C)_-,
\end{align}
where
$\EuScript{T}^{\circ}_{\ell}(C)_{\varepsilon}$
($\varepsilon=\pm$)
is the subring of $\EuScript{T}^{\circ}_{\ell}(C)$
generated by
 $T^{(a)}_m(u)$
$((a,m,u):\mathbf{Q}_{\varepsilon})$.

Similarly,
for a triplet $(a,m,u)\in \mathcal{I}_{\ell}$,
we set the parity conditions $\mathbf{Q}'_{+}$ 
as follows.
\begin{align}
\mathbf{Q}'_+&=
\begin{cases}
\mbox{\rm (o+)  $m+tu$ is odd}& \mbox{$d_a$ is odd, $a\in I_+$}\\
\mbox{\rm (o$-$) $m+tu$ is even}& \mbox{$d_a$ is odd, $a\in I_-$}\\
\mbox{\rm (e+)  $tu$ is odd}& \mbox{$d_a$ is even, $a\in I_+$}\\
\mbox{\rm (e$-$) $tu$ is even}& \mbox{$d_a$ is even, $a\in I_-$}\\
\end{cases}
\end{align}
We have
\begin{align}
(a,m,u):\mathbf{Q}'_+
\ \Longleftrightarrow \
\textstyle (a,m,u\pm\frac{d_a}{t}):\mathbf{Q}_+.
\end{align}
Let $\mathbf{Q}'_{-}$ be the negation of $\mathbf{Q}'_{+}$.
Then, we have the parity decomposition
\begin{align}
%\label{eq:Tfact}
\EuScript{Y}^{\circ}_{\ell}(C)
\simeq
\EuScript{Y}^{\circ}_{\ell}(C)_+
\times
\EuScript{Y}^{\circ}_{\ell}(C)_-,
\end{align}
where
$\EuScript{Y}^{\circ}_{\ell}(C)_{\varepsilon}$
($\varepsilon=\pm$)
is the subring of $\EuScript{Y}^{\circ}_{\ell}(C)$
generated by
 $Y^{(a)}_m(u)$,  $1+Y^{(a)}_m(u)$
$((a,m,u):\mathbf{Q}'_{\varepsilon})$.

\subsection{Construction of quiver $Q_{\ell}(C)$}
\label{subsect:const}

Let us construct a quiver $Q_{\ell}(C)$ for $C$
and $\ell$.
We do it in two steps.
In Step 1, to each adjacent pair
$(a,b)$ of the Dynkin diagram $X(C)$
 we attach a certain quiver $Q(a,b)$.
In Step 2, these quivers are `patched' at each vertex.

{\bf Step 1.} $Q(a,b)$.

Recall that $t$ is the one in \eqref{eq:t1}.
Below suppose that $a$ and $b$ are adjacent in $X(C)$
and $d_a\geq d_b$.

{\bf Case (i). $d_a$ is odd  and $d_b=1$.}
(a) The case $a\in I_-$.
We set the quiver $Q(a,b)$ by the quiver $Q_{\ell'}(M_{t'})$
in Sec.~\ref{subsect:quiverodd}
with $t'=d_a$ and $\ell'=t\ell/d_a$.
We assign $+$/$-$ as in Sec.~\ref{subsect:quiverodd}.
(We do not need to assign $\bullet$/$\circ$ here.)
\par
(b) The case $a\in I_+$.
We set the quiver $Q(a,b)$ by the quiver
$Q(1)$ obtained from $Q(0)=Q_{\ell'}(M_{t'})$
in Sec.~\ref{subsect:quiverodd}
with $t'=d_a$ and $\ell'=t\ell/d_a$.
We assign $+$/$-$ in the opposite way to  Sec.~\ref{subsect:quiverodd}.

{\bf Case (ii). $d_a=d_b$, and $d_a$ is odd  and not 1.}
We can assume that $a\in I_-$ and $b\in I_+$.
We set the quiver $Q(a,b)$ as a disjoint union
of quivers $Q_1$,\dots,$Q_{d_a}$ specified as follows.
The quivers $Q_{1}$, $Q_{3}$,\dots, $Q_{d_a}$
are the quiver $Q_{\ell'}(M_{t'})$
in Sec.~\ref{subsect:quiverodd}
with $t'=1$ and $\ell'=t\ell/d_a$.
We assign $+$/$-$ as in Sec.~\ref{subsect:quiverodd}.
The quivers $Q_{2}$, $Q_{4}$,\dots, $Q_{d_a-1}$
are the opposite quiver of $Q_{\ell'}(M_{t'})$
in Sec.~\ref{subsect:quiverodd}
with $t'=1$ and $\ell'=t\ell/d_a$.
We assign $+$/$-$ in the opposite way to Sec.~\ref{subsect:quiverodd}.

{\bf Case (iii). $d_a$ is even  and $d_b=1$.}
(a) The case $a\in I_+$ and $c_a=\alpha$.
We set the quiver $Q(a,b)$ by the quiver $Q_{\ell'}(M_{t'})$
in Sec.~\ref{subsect:quivereven}
with $t'=d_a$ and $\ell'=t\ell/d_a$.
We assign $+$/$-$ as in Sec.~\ref{subsect:quivereven}.
\par
(b) The case $a\in I_+$ and $c_a=\beta$.
We set the quiver $Q(a,b)$ by the quiver $Q(1)$ 
obtained from
 $Q(0)=Q_{\ell'}(M_{t'})$
in Sec.~\ref{subsect:quivereven}
with $t'=d_a$ and $\ell'=t\ell/d_a$.
For  $\bullet$/$\circ$  in Sec.~\ref{subsect:quivereven},
we assign $+$/$-$ to vertices with $\bullet$ 
as in Sec.~\ref{subsect:quivereven},
while we assign $+$/$-$ to vertices with $\circ$
in the opposite way to Sec.~\ref{subsect:quivereven}.
\par
(c) The case $a\in I_-$ and $c_a=\alpha$.
We set the quiver $Q(a,b)$ by the quiver $Q(-1/t')$ obtained from
 $Q(0)=Q_{\ell'}(M_{t'})$
in Sec.~\ref{subsect:quivereven}
with $t'=d_a$ and $\ell'=t\ell/d_a$.
For  $\bullet$/$\circ$  in Sec.~\ref{subsect:quivereven}.
we assign $+$/$-$ to vertices with $\circ$ 
as in Sec.~\ref{subsect:quivereven},
while we assign $+$/$-$ to vertices with $\bullet$
in the opposite way to Sec.~\ref{subsect:quivereven}.

\par
(d) The case $a\in I_-$ and $c_a=\beta$.
We set the quiver $Q(a,b)$ by the quiver $Q((t'-1)/t')$ obtained from
 $Q(0)=Q_{\ell'}(M_{t'})$
in Sec.~\ref{subsect:quivereven}
with $t'=d_a$ and $\ell'=t\ell/d_a$.
We assign $+$/$-$ in the opposite way to Sec.~\ref{subsect:quivereven}.

{\bf Case (iv). $d_a=d_b$, and $d_a$ is even.}
 We can assume that $c_a=\alpha$ and $c_b=\beta$.
We set the quiver $Q(a,b)$ as a disjoint union
of quivers $Q_1$,\dots,$Q_{d_a}$ specified as follows.
The quivers $Q_{1}$, $Q_{3}$,\dots, $Q_{d_a-1}$
are the quiver $Q_{\ell'}(M_{t'})$
in Sec.~\ref{subsect:quiverodd}
with $t'=1$ and $\ell'=t\ell/d_a$.
We assign $+$/$-$ as in Sec.~\ref{subsect:quiverodd}.
The quivers $Q_{2}$, $Q_{4}$,\dots, $Q_{d_a}$
are the opposite quiver of $Q_{\ell'}(M_{t'})$
in Sec.~\ref{subsect:quiverodd}
with $t'=1$ and $\ell'=t\ell/d_a$.
We assign $+$/$-$ in the opposite way to Sec.~\ref{subsect:quiverodd}.

Throughout Step 1,
we regard the left column(s) of $Q(a,b)$
(with length $t\ell/d_a-1$) as {\em attached
to $a$} and the right column(s) of $Q(a,b)$
(with length $t\ell/d_b-1$) as {\em attached to $b$}.

{\bf Step 2.} $Q_{\ell}(C)$.

 The quiver $Q_{\ell}(C)$ is defined by
patching the above quivers $Q(a,b)$ at each vertex.
Namely, fix $a\in I$,
and take all  $b$'s which are adjacent to $a$.
If $d_a=1$,
we identify the  columns  attached to $a$ in $Q(b,a)$
for all $b$.
If $d_a>1$,
for each  $i=1,\dots,d_a$,
we identify the columns  attached to $a$ in 
the $i$th quivers $Q_i$ of $Q(a,b)$ or $Q(b,a)$
(depending on the sign and  color of $a$) for all $b$.
(For Cases (i) and (ii), $Q_i$ appears
in the construction of $Q_{\ell}(M_t)$.)

Some basic examples are given below.

\begin{Example}
\label{ex:1}
The two examples below mostly clarify the situation
involving Cases (i) and (ii).
\par
(1) Let $C$ be the Cartan matrix with the following 
Dynkin diagram.
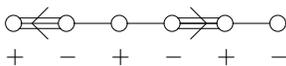
\begin{align*}

\end{align*}
The corresponding quiver $Q_{\ell}(C)$ for $\ell=3$ is
given as follows.
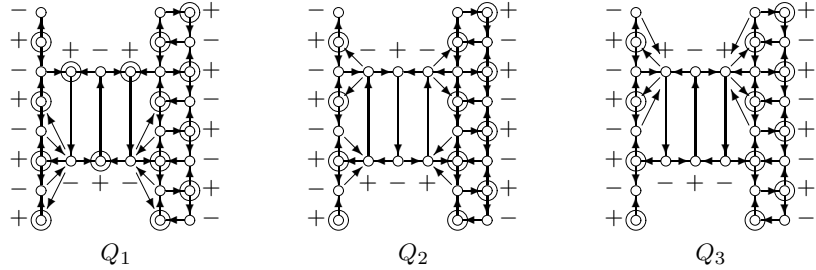
\begin{align*}
\setlength{\unitlength}{0.75pt}
%
\end{align*}
Here, the long columns at the same horizontal positions 
 in three quivers are identified
with each other.
The encircled vertices are the mutation points
at $Q(0)=Q_{\ell}(C)$, which will be described in the next
subsection.
The same remark applies below.
\par
(2) Let $C$ be the Cartan matrix with the following 
Dynkin diagram.
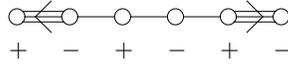
\begin{align*}
%
\end{align*}
The corresponding quiver $Q_{\ell}(C)$ for $\ell=3$ is
given as follows.
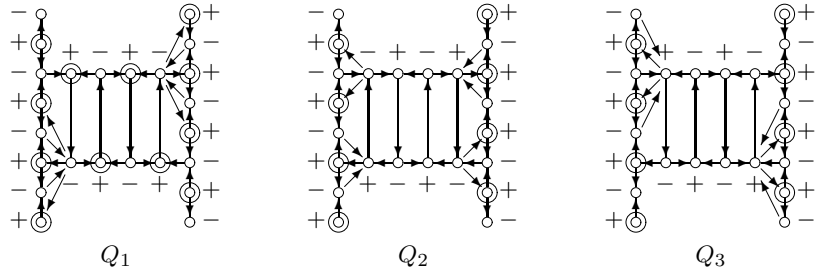
\begin{align*}
\setlength{\unitlength}{0.75pt}
%
\end{align*}
\end{Example}

\begin{Example}
\label{ex:2}
The four examples below mostly clarify the situation
involving Cases (iii) and (iv).
\par
(1) Let $C$ be the Cartan matrix with the following 
Dynkin diagram.
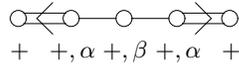
\begin{align*}
%
\end{align*}
The corresponding quiver $Q_{\ell}(C)$ for $\ell=3$ is
given as follows.
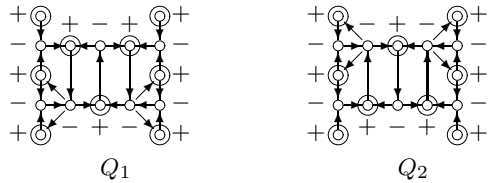
\begin{align*}
\setlength{\unitlength}{0.75pt}
%
\end{align*}
\par
(2) Let $C$ be the Cartan matrix with the following 
Dynkin diagram.
\begin{align*}
%
\end{align*}
\par
\par
(3) Let $C$ be the Cartan matrix with the following 
Dynkin diagram.
\begin{align*}
%
\end{align*}
\par
\par
(4) Let $C$ be the Cartan matrix with the following 
Dynkin diagram.
\begin{align*}
%
\end{align*}
\end{Example}

\subsection{Mutation sequence}

We set $Q(0)=Q_{\ell}(C)$
and define a periodic sequence of mutations
of quivers
\begin{align}
\label{eq:Qmuttree}
Q(0)\
\displaystyle
\mathop{\longleftrightarrow}^{\mu(0)}
\
\textstyle
 Q(\frac{1}{t})
\
\displaystyle
\mathop{\longleftrightarrow}^{\mu(\frac{1}{t})}
\
\textstyle
 Q(\frac{2}{t})
\
\displaystyle
\mathop{\longleftrightarrow}^{\mu(\frac{2}{t})}
\
\cdots
\
\displaystyle
\mathop{\longleftrightarrow}^{\mu(\frac{2t-1}{t})}
\
 Q(2)=Q(0)
\end{align}
by patching the ones in \eqref{eq:GB2} and \eqref{eq:BB2}.
Let $M_a(k/t)$ ($a\in I$, $k=0,\dots,2t-1$)
be the set of the mutation points of $\mu(k/t)$
 in the columns attached to $a$.
It is defined as follows.
(Below we  use the assignment of $+/-$
specified in Sec.~\ref{subsect:const}.
Also we use the similar notations in
the ones in \eqref{eq:GB2} and \eqref{eq:BB2},
e.g., $\mathbf{I}^{a}_{+,i}$ denotes 
the set of vertices
in the column attached to $a$ 
  of $i$th quiver $Q_i$ 
with property $+$.
)

(i) $d_a$: odd. (cf.~\eqref{eq:GB2})
\begin{align}
\begin{split}
&M_a(0)=\mathbf{I}^{a}_{+,1},\
\textstyle
M_a(\frac{1}{t})=\mathbf{I}^{a}_{+,d_a-1},\
M_a(\frac{2}{t})=\mathbf{I}^{a}_{+,3},
\dots,\\
&
\textstyle
M_a(\frac{2d_a-2}{t})=\mathbf{I}^{a}_{-,2},\
M_a(\frac{2d_a-1}{t})=\mathbf{I}^{a}_{-,d_a},\
M_a(\frac{2d_a}{t})=M_a(0), \dots
\end{split}
\end{align}
In particular, for $d_a=1$,
\begin{align}
\textstyle
M_a(0)=\mathbf{I}^{a}_{+,1},\
M_a(\frac{1}{t})=\mathbf{I}^{a}_{-,1},\
M_a(\frac{2}{t})=M_0(0), \dots
\end{align}

(ii) $d_a$: even, $a\in I_+$ (cf.~\eqref{eq:BB2})
\begin{align}
\label{eq:mut11}
\begin{split}
&
\textstyle
M_a(0)=\mathbf{I}^{a}_{+,1,d_a},\
M_a(\frac{1}{t})=\emptyset,\
M_a(\frac{2}{t})=\mathbf{I}^{a}_{+,3,d_a-2},
\dots,\\
&
\textstyle
M_a(\frac{2d_a-2}{t})=\mathbf{I}^{a}_{-,d_a-1,2},\
M_a(\frac{2d_a-1}{t})=\emptyset,\
M_a(\frac{2d_a}{t})=M_0(0), \dots
\end{split}
\end{align}

(iii) $d_a$: even, $a\in I_-$ (cf.~\eqref{eq:BB2})
\begin{align}
\label{eq:mut12}
\begin{split}
&
\textstyle
M_a(0)=\emptyset,\
M_a(\frac{1}{t})=\mathbf{I}^{a}_{+,1,d_a},\
M_a(\frac{2}{t})=\emptyset,\
M_a(\frac{3}{t})=\mathbf{I}^{a}_{+,3,d_a-2},
\dots,\\
&
\textstyle
M_a(\frac{2d_a-2}{t})=\emptyset,\
M_a(\frac{2d_a-1}{t})=\mathbf{I}^{a}_{-,d_a-1,2},\
M_a(\frac{2d_a}{t})=M_0(0), \dots
\end{split}
\end{align}

\subsection{T-system, Y-system, and cluster algebra}
\label{subsect:treeTY}

Now it is straightforward to repeat the formulation
in Secs.~\ref{sect:todd} and \ref{sect:teven}.
The compatibility of mutations is the only issue,
but it has been already taken care of in the construction
of $Q_{\ell}(C)$ as self-explained in
Examples \ref{ex:1} and \ref{ex:2}.

Let $\mathbf{I}$ be the index set of the quiver
$Q_{\ell}(C)$.
Let $B$ the skew-symmetric matrix corresponding
to $Q_{\ell}(C)$.
Using the sequence \eqref{eq:Qmuttree},
we define cluster variables $x_{\mathbf{i}}(u)$
and coefficients $y_{\mathbf{i}}(u)$ ($\mathbf{i}\in \mathbf{I},
u\in \frac{1}{t}\mathbb{Z}$) as before.
Define the T-subalgebra $\mathcal{A}_T(B,x)$
and Y-subgroup $\mathcal{A}_T(B,y)$
as parallel to
Definitions \ref{def:T-sub} and \ref{def:Y-sub}.

Repeating the same argument as before,
 we obtain the conclusion of this section.
\begin{Theorem}
\label{thm:GYisotree}
Let $C$ be any tamely laced and indecomposable
Cartan matrix whose Dynkin diagram is a tree.
Then, the ring $\EuScript{T}^{\circ}_{\ell}(C)_+$ is isomorphic to
$\mathcal{A}_T(B,x)$. 
The group $\EuScript{Y}^{\circ}_{\ell}(C)_+$ is isomorphic to
$\mathcal{G}_Y(B,y)$. 
\end{Theorem}

\section{Cluster algebraic formulation: General case}
\label{sect:general}

It is easy to extend Theorem 
\ref{thm:GYisotree} to any 
tamely laced  Cartan matrix
$C$ with suitable modification.
Due to the lack of the space,
we concentrate on describing the construction
of the quiver $Q_{\ell}(C)$.
Throughout the section we assume that
$C$ is a tamely laced and indecomposable
Cartan matrix.

Before starting, we introduce some preliminary definitions.
We call a subdiagram $Y$  of $X(C)$ an {\em even block},
if $Y$ is an maximal indecomposable subdiagram of $X(C)$
such that $d_a$ of each vertex $a$ of $Y$ is {\em even}.
Due to the condition \eqref{eq:Ccond1}, $d_a$ is constant
for any vertex $a$ of $Y$.
Below we suppose that $X(C)$ has
$n$ even blocks $Y_1$,\dots, $Y_n$. ($n$ may be zero.)
Let $X'(C)$ be the diagram obtained from $X(C)$
by shrinking each even block $Y_i$
into a vertex `$\otimes $' while keeping any
line  from $Y_i$ to its outside.
For example, for the following $X(C)$
\begin{align}
\label{eq:diagram1}
\begin{picture}(100,10)(0,0)
\put(0,0){\circle{6}}
\put(20,0){\circle{6}}
\put(40,0){\circle{6}}
\put(60,0){\circle{6}}
\put(80,0){\circle{6}}
\put(100,0){\circle{6}}
%\put(47,0){\circle*{1}}
%\put(50,0){\circle*{1}}
%\put(53,0){\circle*{1}}
\drawline(2,-2)(18,-2)
\drawline(2,2)(18,2)
%\drawline(3,0)(17,0)
\drawline(7,0)(13,6)
\drawline(7,0)(13,-6)
\drawline(23,0)(37,0)
\drawline(43,0)(57,0)
\drawline(82,-2)(98,-2)
\drawline(82,2)(98,2)
\drawline(87,6)(93,0)
\drawline(87,-6)(93,0)
\drawline(63,0)(77,0)
\end{picture}
\end{align}
$X'(C)$ is given by
\begin{align}
\label{eq:diagram2}
\begin{picture}(40,10)(0,0)
\put(0,0){\circle{6}}
\put(20,0){\circle{6}}
\put(40,0){\circle{6}}
\drawline(2,-2)(18,-2)
\drawline(2,2)(18,2)
\drawline(7,0)(13,6)
\drawline(7,0)(13,-6)
\drawline(22,2)(18,-2)
\drawline(22,-2)(18,2)
\drawline(22,-2)(38,-2)
\drawline(22,2)(38,2)
\drawline(27,6)(33,0)
\drawline(27,-6)(33,0)
\end{picture}
\end{align}

\subsection{The case $X'(C)$ is bipartite}
\label{subsect:bipartite}

Let us assume that $X'(C)$ is {\em bipartite},
i.e., it contains  no odd cycle.

First, consider the case when all the even blocks $Y_1$, \dots, $Y_n$
are also bipartite.
Then, $X(C)$ admits a decomposition
and a coloring of  $I$ satisfying
Conditions (I)--(III) in Sec.~\ref{subsect:parity},
and one can construct $Q_{\ell}(C)$ as in Sec.~\ref{subsect:const}.

Next, consider the case when
some of the even blocks, say,
 $Y_1$, \dots, $Y_k$ are nonbipartite.
Then, 
$X(C)$ does not admit a coloring of $I$
satisfying
Condition (III) in Sec.~\ref{subsect:parity}.
Following Ref.~\refcite{Kuniba09},
we define the  {\em bipartite double} $Y^{\#}$
of any tamely laced Dynkin diagram $Y$ as follows.
Let $J$ be the vertex set of $Y$.
The vertex set  $J^{\#}$ of $Y^{\#}$ is the disjoint union
$J^{\#}=J_+\sqcup J_-$, where $J_+=\{j_+\mid j\in J\}$
and $J_-=\{j_-\mid j\in J\}$;
furthermore, we write a line (or multiple line with arrow)  in $Y^{\#}$ from
$i_+$ to $j_-$ and also from
$i_-$ to $j_+$ if and only if there is a line (or multiple line with arrow) 
from $i$ to $j$ in $Y$.
Let $\tilde{X}(C)$ be the diagram obtained from 
$X(C)$ by replacing each nonbipartite even block $Y_i$
($i=1,\dots,k$) with its bipartite double $Y_i^{\#}$,
while connecting $i_{\pm}$ in $Y_i$ to any vertex $j$ outside $Y_i$ 
by a line (or  multiple line with arrow) if and only if $i$ and $j$ are
connected in $X(C)$ by a line (or multiple line with arrow).
The diagram $\tilde{X}(C)$ now admits a decomposition
and coloring  satisfying
Conditions (I)--(III) in Sec.~\ref{subsect:parity}.
See Fig.~\ref{fig:double1} for an example.
Then, we repeat the construction of 
the quiver $Q_{\ell}(C)$ in Sec.~\ref{subsect:const}
for the diagram $\tilde{X}(C)$ with the following modification:
{\em In Step 1 of Sec.~\ref{subsect:const},
in Cases (iii) and (iv),
we only take the $d_a/2$ subquivers $Q_1$, $Q_3$,\dots, $Q_{d_a-1}$
for those $Q(a,b)$ involving the vertices of
$Y_1^{\#}$, \dots, $Y_k^{\#}$.}
We write the resulted quiver as $Q_{\ell}(C)$.
Accordingly, we also replace 
$\mathbf{I}^{a}_{+,1,d_a}$, $\mathbf{I}^{a}_{+,3,d_a-2}$,
\dots in \eqref{eq:mut11} and \eqref{eq:mut12}
 with 
$\mathbf{I}^{a}_{+,1}$, $\mathbf{I}^{a}_{+,3}$, \dots.
The rest are defined in the same way as in
Sect.~\ref{subsect:treeTY}.

\begin{figure}[t]
\begin{picture}(180,73)(-70,-34)
%
% A_r
%
% B_r
\put(0,0)
{
\put(0,0){\circle{6}}
\put(20,0){\circle{6}}
\put(40,0){\circle{6}}
\put(0,-20){\circle{6}}
\put(40,-20){\circle{6}}
\drawline(27,6)(33,0)
\drawline(27,-6)(33,0)
\drawline(2,-2)(18,-2)
\drawline(2,2)(18,2)
\drawline(7,0)(13,6)
\drawline(7,0)(13,-6)
\drawline(22,-2)(38,-2)
\drawline(22,2)(38,2)
\drawline(3,-20)(37,-20)
%\drawline(0,-3)(0,-17)
%\drawline(40,-3)(40,-17)
\drawline(2,-18)(18,-2)
\drawline(38,-18)(22,-2)
\put(-2,12){\small $1$}
\put(18,12){\small $2$}
\put(38,12){\small $3$}
\put(-2,-33){\small $5$}
\put(38,-33){\small $4$}
%\put(-3,10){\small $+$}
%\put(17,10){\small $-$}
%\put(37,-15){\small $+$}
%\put(57,-25){\small $\alpha$}
%\put(77,-25){\small $\beta$}
%
}
\put(100,0)
{
\put(0,20){\circle{6}}
\put(20,0){\circle{6}}
\put(40,20){\circle{6}}
\put(60,0){\circle{6}}
\put(80,20){\circle{6}}
\put(0,-20){\circle{6}}
\put(40,-20){\circle{6}}
\put(80,-20){\circle{6}}
\drawline(3,20)(37,20)
\drawline(23,1)(39,17)
\drawline(21,3)(37,19)
\drawline(27,7)(27,13)
\drawline(27,7)(33,7)
\drawline(57,1)(41,17)
\drawline(59,3)(43,19)
\drawline(53,7)(53,13)
\drawline(53,7)(47,7)
\drawline(43,-19)(59,-3)
\drawline(41,-17)(57,-1)
\drawline(53,-7)(53,-13)
\drawline(53,-7)(47,-7)
\drawline(37,-19)(21,-3)
\drawline(39,-17)(23,-1)
\drawline(27,-7)(27,-13)
\drawline(27,-7)(33,-7)
\drawline(43,20)(77,20)
\drawline(3,-20)(37,-20)
\drawline(43,-20)(77,-20)
\drawline(0,17)(0,-17)
\drawline(80,17)(80,-17)
\put(-2,34){\small $4_-$}
\put(38,34){\small $2_+$}
\put(78,34){\small $5_-$}
\put(-2,-39){\small $5_+$}
\put(38,-39){\small $2_-$}
\put(78,-39){\small $4_+$}
\put(10,4){\small $1$}
\put(66,4){\small $3$}
\put(9,-5){\small $+$}
\put(65,-5){\small $+$}
\put(-8,25){\small $+,\beta$}
\put(32,25){\small $+,\alpha$}
\put(72,25){\small $+,\beta$}
\put(-8,-29){\small $+,\alpha$}
\put(32,-29){\small $+,\beta$}
\put(72,-29){\small $+,\alpha$}
%\put(17,10){\small $-$}
%\put(37,-15){\small $+$}
%\put(37,-8){\small $\alpha$}
%\put(37,-16){\small $\alpha$}
%\put(77,-25){\small $\beta$}
%
}
\end{picture}
\caption{Example of Dynkin diagram $X(C)$ (left)
and $\tilde{X}(C)$ (right).}
\label{fig:double1}
\end{figure}

Now we have the first main result of the paper.
\begin{Theorem}
\label{thm:GYisobi}
Let $C$ be any tamely laced and indecomposable Cartan matrix
such that $X'(C)$ is bipartite.
Let $B$ the skew-symmetric matrix corresponding
to the quiver $Q_{\ell}(C)$ defined above.
Then, the ring $\EuScript{T}^{\circ}_{\ell}(C)_+$ is isomorphic to
$\mathcal{A}_T(B,x)$. 
The group $\EuScript{Y}^{\circ}_{\ell}(C)_+$ is isomorphic to
$\mathcal{G}_Y(B,y)$. 
\end{Theorem}

\subsection{The case $X'(C)$ is nonbipartite}

\begin{figure}[t]
\begin{picture}(200,62)(-60,-40)
%
% A_r
%
% B_r
\put(0,0)
{
\put(0,0){\circle{6}}
\put(20,0){\circle{6}}
\put(40,0){\circle{6}}
\put(60,0){\circle{6}}
\put(0,-20){\circle{6}}
\put(60,-20){\circle{6}}
\drawline(47,6)(53,0)
\drawline(47,-6)(53,0)
\drawline(2,-2)(18,-2)
\drawline(2,2)(18,2)
\drawline(7,0)(13,6)
\drawline(7,0)(13,-6)
\drawline(23,0)(37,0)
\drawline(42,-2)(58,-2)
\drawline(42,2)(58,2)
\drawline(3,-20)(57,-20)
\drawline(0,-3)(0,-17)
\drawline(60,-3)(60,-17)
\put(-2,12){\small $1$}
\put(18,12){\small $2$}
\put(38,12){\small $3$}
\put(58,12){\small $4$}
\put(-2,-33){\small $6$}
\put(58,-33){\small $5$}
%\put(-3,10){\small $+$}
%\put(17,10){\small $-$}
%\put(37,-15){\small $+$}
%\put(57,-25){\small $\alpha$}
%\put(77,-25){\small $\beta$}
%
}
\put(100,0)
{
\put(0,0){\circle{6}}
\put(20,0){\circle{6}}
\put(40,0){\circle{6}}
\put(60,0){\circle{6}}
\put(80,0){\circle{6}}
\put(100,0){\circle{6}}
\put(0,-20){\circle{6}}
\put(20,-20){\circle{6}}
\put(40,-20){\circle{6}}
\put(60,-20){\circle{6}}
\put(80,-20){\circle{6}}
\put(100,-20){\circle{6}}
\drawline(3,0)(17,0)
\drawline(67,6)(73,0)
\drawline(67,-6)(73,0)
\drawline(22,-2)(38,-2)
\drawline(22,2)(38,2)
\drawline(27,0)(33,6)
\drawline(27,0)(33,-6)
\drawline(43,0)(57,0)
\drawline(62,-2)(78,-2)
\drawline(62,2)(78,2)
\drawline(83,0)(97,0)
\drawline(3,-20)(17,-20)
\drawline(67,-14)(73,-20)
\drawline(67,-26)(73,-20)
\drawline(22,-22)(38,-22)
\drawline(22,-18)(38,-18)
\drawline(27,-20)(33,-14)
\drawline(27,-20)(33,-26)
\drawline(43,-20)(57,-20)
\drawline(62,-22)(78,-22)
\drawline(62,-18)(78,-18)
\drawline(83,-20)(97,-20)
\drawline(0,-3)(0,-17)
\drawline(100,-3)(100,-17)
\put(-2,17){\small $6_-$}
\put(18,17){\small $1_+$}
\put(38,17){\small $2$}
\put(58,17){\small $3$}
\put(78,17){\small $4+$}
\put(98,17){\small $5_-$}
\put(-2,-43){\small $5_+$}
\put(18,-43){\small $4_-$}
\put(38,-43){\small $3$}
\put(58,-43){\small $2$}
\put(78,-43){\small $1_-$}
\put(98,-43){\small $6_+$}
\put(-3,9){\small $-$}
\put(17,9){\small $+$}
\put(32,9){\small $+,\alpha$}
\put(52,9){\small $+,\beta$}
\put(77,9){\small $+$}
\put(97,9){\small $-$}
\put(-3,-33){\small $+$}
\put(17,-33){\small $-$}
\put(32,-33){\small $-,\alpha$}
\put(52,-33){\small $-,\beta$}
\put(77,-33){\small $-$}
\put(97,-33){\small $+$}
%\put(17,10){\small $-$}
%\put(37,-15){\small $+$}
%\put(37,-8){\small $\alpha$}
%\put(37,-16){\small $\alpha$}
%\put(77,-25){\small $\beta$}
%
}
\end{picture}
\caption{Example of Dynkin diagram $X(C)$ (left)
and $\tilde{X}(C)$ (right).}
\label{fig:double}
\end{figure}

Let us assume that  $X'(C)$ is {\em nonbipartite}.
Then,  $X(C)$ does not admit a decomposition of
$I$ satisfying
Conditions (I) and (II) in Sec.~\ref{subsect:parity};
consequently,
neither $\EuScript{T}^{\circ}_{\ell}(C)$
nor $\EuScript{Y}^{\circ}_{\ell}(C)$
admits the parity decomposition.

First, consider the case when all the even blocks $Y_1$, \dots, $Y_n$
of $X(C)$ are bipartite.
We take the bipartite double  $X'(C)^{\#}$ 
of $X'(C)$.
Then, in  $X'(C)^{\#}$,
 restore each even block of $X(C)$, which appears
twice in $X'(C)^{\#}$,  in place of $\otimes$.
The resulted diagram $\tilde{X}(C)$ now admits a decomposition
and a coloring satisfying
Conditions (I)---(III) in Sec.~\ref{subsect:parity}.
See Fig.~\ref{fig:double} for an example.
Now we repeat the construction of the quiver
$Q_{\ell}(C)$ in Sec.~\ref{subsect:const}
for the diagram $\tilde{X}(C)$.
We write the resulted quiver as $Q_{\ell}(C)$.

Next, consider the case when some of the even blocks
of $X(C)$,
say, $Y_1$, \dots, $Y_k$ are nonbipartite.
Then, in the above construction of 
$\tilde{X}(C)$, we further replace
each nonbipartite even block $Y_i$ ($i=1,\dots,k$)
with its bipartite double $Y_i^{\#}$
as in Sec.~\ref{subsect:bipartite}.
We write the resulted diagram as $\tilde{X}(C)$.
Then, repeat the construction of the quiver $Q_{\ell}({C})$
in Sec.~\ref{subsect:bipartite}
for the diagram $\tilde{X}(C)$.
We write the resulted quiver as $Q_{\ell}(C)$.

The rest are defined in the same way as before.
Then, as in the simply laced case \cite{Kuniba09},
we have the counterpart of Theorem \ref{thm:GYisobi},
which is the second main result of the paper.

\begin{Theorem}
\label{thm:GYisononbi}
Let $C$ be any tamely laced and indecomposable Cartan matrix
such that $X'(C)$ is nonbipartite.
Let $B$ the skew-symmetric matrix corresponding
to the quiver $Q_{\ell}(C)$ defined above.
Then, the ring $\EuScript{T}^{\circ}_{\ell}(C)$ is isomorphic to
$\mathcal{A}_T(B,x)$. 
The group $\EuScript{Y}^{\circ}_{\ell}(C)$ is isomorphic to
$\mathcal{G}_Y(B,y)$. 
\end{Theorem}

\section*{Acknowledgments}
It is my great pleasure to thank
Professor Tetsuji Miwa
on the occasion of his sixtieth birthday
 for his generous
support 
 and continuous interest in my works
through many years.
I thank Rei Inoue, Osamu Iyama, Bernhard Keller,
Atsuo Kuniba, and Junji Suzuki
for sharing their insights in the preceding joint works.

%\bibliographystyle{ws-procs9x6}
%\bibliography{../biblist/biblist}
%\end{document}

\end{document}